\documentclass[reqno]{amsart}
\usepackage{geometry}
\usepackage[utf8]{inputenc}
\usepackage{esint} %
\usepackage{stmaryrd} %
\SetSymbolFont{stmry}{bold}{U}{stmry}{m}{n}  %
\usepackage{xfrac} %
\usepackage{comment}

\usepackage{mathtools,microtype,lmodern,amssymb}
\mathtoolsset{centercolon}
\usepackage[dvipsnames]{xcolor}
\usepackage{enumerate}
\usepackage{tikz}
\usetikzlibrary {arrows.meta} 
\usepackage[colorlinks=true,citecolor=Cyan]{hyperref}

\DeclarePairedDelimiter{\abs}{\lvert}{\rvert}
\DeclarePairedDelimiter{\norm}{\lVert}{\rVert}
\DeclarePairedDelimiter{\bra}{(}{)}
\DeclarePairedDelimiter{\pra}{[}{]}
\DeclarePairedDelimiter{\set}{\{}{\}}
\DeclarePairedDelimiter{\skp}{\langle}{\rangle}

\newcommand{\Z}{\mathbb{Z}}
\newcommand{\R}{\mathbb{R}}
\newcommand{\Rp}{\mathbb{R}_{>0}}
\newcommand{\Rnn}{\mathbb{R}_{\ge0}}
\newcommand{\crc}{{{\mathbb S}^1}}

\newcommand{\Leb}{{\mathcal L}}
\newcommand{\interval}{{\mathcal I}}

\newcommand{\dst}{{\mathbb{D}}}
\newcommand{\ddst}{\mathbb{D}^\delta}

\newcommand{\Prob}{\mathcal{P}}
\newcommand{\radon}{{\mathfrak M}}
\newcommand{\dprb}{\mathcal{P}^\delta}

\newcommand{\dn}{\mathrm{d}}
\newcommand{\dd}{\,\mathrm{d}}
\DeclareMathAlphabet{\mathup}{OT1}{\familydefault}{m}{n}
\newcommand{\dx}[1]{\mathop{}\!\mathup{d} #1} %

\newcommand{\fnc}{\mathcal{E}}
\newcommand{\dfnc}{\mathcal{E}^\delta}
\newcommand{\ent}{\mathcal{H}}
\newcommand{\fish}{\mathcal{F}}
\newcommand{\lyp}{\mathcal{L}}
\newcommand{\fnclog}{\mathcal{L}}
\newcommand{\dent}{\mathcal{H}^\delta}
\newcommand{\dfish}{\mathcal{F}^\delta}
\newcommand{\dlyp}{\mathcal{L}^\delta}
\newcommand{\eps}{\varepsilon}

\newcommand{\sgrp}{{\mathfrak S}}

\newcommand{\frc}[2]{\pra*{\frac{#1^2}{#2}}}

\newcommand{\AC}{\textup{AC}}

\newcommand{\loc}{{\textup{loc}}}

\newcommand{\mob}{\mathbf{M}}
\newcommand{\mobb}{\mathbf{m}}
\newcommand{\logmean}{\boldsymbol{\Lambda}}
\newcommand{\geom}{\boldsymbol{\Gamma}}

\newcommand{\ddff}{\Delta^{\!\delta}}

\newcommand{\hell}{\mathbb{H}}
\newcommand{\dhell}{\mathbb{H}^\delta}
\newcommand{\wass}{\mathbb{W}}
\newcommand{\dwass}{\mathbb{W}^\delta}
\newcommand{\welo}{\mathfrak{w}}
\newcommand{\curves}{{\mathcal C}}

\newcommand{\cto}{[N]}
\newcommand{\m}{\mathbf{M}}
\newcommand{\moment}{{\mathfrak m}}

\newcommand{\act}{\mathfrak{a}}
\newcommand{\action}{\mathfrak{A}}
\newcommand{\daction}{\mathfrak{A}^\delta}

\newcommand{\connect}{\mathcal{CE}}

\newcommand{\Ifnc}{{\mathcal I}}
\newcommand{\Jfnc}{{\mathcal J}}
\newcommand{\ddd}{{\mathrm D}}
\newcommand{\mf}{{\mathcal M}}
\newcommand{\tg}{\mathrm{T}}
\newcommand{\grd}{\operatorname{grad}}
\newcommand{\ons}{{\mathbb K}}
\newcommand{\dprbplus}{\mathcal{P}^\delta_{>0}}

\newcommand{\pid}{\Pi^{\smash\delta}}

\newcommand{\rhon}{\rho^{\delta}}
\newcommand{\dnorm}[2]{\left\| #2 \right\|_{#1}}

\newcommand{\tnorm}[1]{{\left\vert\kern-0.25ex\left\vert\kern-0.25ex\left\vert #1 
    \right\vert\kern-0.25ex\right\vert\kern-0.25ex\right\vert}}

\newtheorem{manualtheoreminner}{Theorem}
\newenvironment{manualtheorem}[1]{%
	\renewcommand\themanualtheoreminner{#1}%
	\manualtheoreminner
}{\endmanualtheoreminner}
\newtheorem{manualresultinner}{Result}
\newenvironment{manualresult}[1]{%
	\renewcommand\themanualresultinner{#1}%
	\manualresultinner
}{\endmanualresultinner}    
\newtheorem{theorem}{Theorem}
\newtheorem{proposition}[theorem]{Proposition}
\newtheorem{lemma}[theorem]{Lemma}

\newtheorem{definition}[theorem]{Definition}
\newtheorem{remark}[theorem]{Remark}

\date{\today}

\allowdisplaybreaks

\title[A structure preserving discretization for the DLSS equation]{A structure preserving discretization for the Derrida-Lebowitz-Speer-Spohn equation \\ based on diffusive transport}
\author{Daniel Matthes}
\address{Daniel Matthes -- Zentrum Mathematik, TU München, Boltzmannstrasse 3, D-85748 Garching, Germany}
\email{matthes@ma.tum.de}
\author{Eva-Maria Rott}
\address {Eva-Maria Rott -- Zentrum Mathematik, TU München, Boltzmannstrasse 3, D-85748 Garching, Germany}
\email{eva-maria.rott@tum.de}
\author{Giuseppe Savaré}
\address{Giuseppe Savar\'e -- Bocconi University,
	Department of Decision Sciences and BIDSA, Via Roentgen 1, 20136 Milano (Italy)}
\email{giuseppe.savare@unibocconi.it}
\author{André Schlichting}
\address{André Schlichting -- Institute for Applied Analysis, Ulm University, Helmholtzstrasse 18, D-89081 Ulm, Germany}
\email{andre.schlichting@uni-ulm.de}

\newcommand{\change}[1]{{\color{blue}#1}}
\renewcommand{\change}[1]{#1}

\begin{document}

\begin{abstract}
	We propose a spatial discretization of the fourth-order nonlinear DLSS equation on the circle. Our choice of discretization is motivated by a novel gradient flow formulation with respect to a metric that generalizes martingale transport. The discrete dynamics inherits this gradient flow structure, and in addition further properties, such as an alternative gradient flow formulation in the Wasserstein distance, contractivity in the Hellinger distance, and monotonicity of several Lypunov functionals. Our main result is the convergence in the limit of vanishing mesh size. The proof relies an a discrete version of a nonlinear functional inequality between integral expressions involving second order derivatives. 
\end{abstract}

\thanks{\emph{Acknowledgement:} 
The authors would like to thank the anonymous referees for their very detailed and constructive reports, which helped to improve the quality of the manuscript.	
AS thanks Jasper Hoeksema pointing out the generalized gradient structure in Remark~\ref{rem:GGF}.\\
\emph{Funding:}  DM's and ER's research is supported by the DFG Collaborative Research Center TRR 109, ``Discretization in Geometry and Dynamics.''
GS gratefully acknowledge the support of the Institute for Advanced Study of the Technical University of Munich, funded by the German Excellence Initiative.
GS has been supported by the MIUR-PRIN 202244A7YL project ``Gradient Flows and Non-Smooth Geometric Structures with Applications to Optimization and Machine Learning'' and by the INdAM Project 2024.
AS is supported by the Deutsche Forschungsgemeinschaft (DFG, German Research Foundation) under Germany's Excellence Strategy EXC 2044 --390685587, Mathematics M\"unster: Dynamics--Geometry--Structure.}

\maketitle

\section{Introduction and main results}
In the article at hand, we devise and analyze a spatial discretization of the Derrida-Lebowitz-Speer-Spohn (DLSS) equation, also known as quantum drift diffusion (QDD) equation,
\begin{equation}\label{eq:DLSS}
    \partial_t \rho 
    = - \partial_{xx}\bigl( \rho\, \partial_{xx} \log \rho \bigr) 
\end{equation}
with periodic boundary conditions, i.e., on the circle $\crc\cong\R/\Z$. For the relevance of \eqref{eq:DLSS} in mathematical physics and known analytical results, see Section \ref{sct:history} below. This equation has a variety of remarkable structural properties --- among them several Lyapunov functionals, contractivity properties, and two different gradient flow structures, the second of which is described here for the first time; see Section \ref{sct:QDDprop} for details.

For discretization, we approximate $\rho$ by a piecewise constant function with $N$ values $\rho_\kappa$, $\kappa\in[N]:=\set*{1,\dots,N}$, with respect to a uniform spatial mesh with cells of size $\delta=1/N$ on ${\mathbb S}^1$. Our approximation of~\eqref{eq:DLSS} is given by
\begin{align}
    \label{eq:dDLSS0}
    \dot\rho_\kappa = \frac1{\delta^2}\big(F_{\kappa+1}-2F_\kappa+F_{\kappa-1}\big)
    \qquad\text{where}\qquad
    F_\lambda := -\frac2{\delta^2}\big(\sqrt{\rho_{\lambda+1}\rho_{\lambda-1}}-\rho_\lambda\big), \qquad\kappa,\lambda\in [N].
\end{align}
This is a finite-difference discretization of~\eqref{eq:DLSS}, as can be seen from the expansion
\begin{equation}\label{eq:heuristic:derivation}
	F_\lambda  
	= - 2\frac{\rho_\lambda}{\delta^2} \biggl( \exp\Bigl(\tfrac{1}{2} \bigl( \log \rho_{\lambda+1} + \log \rho_{\lambda-1} - 2 \log \rho_\lambda \bigr) \Bigr) -1 \biggr)
	\approx - \rho \, \partial_{xx} \log \rho  + O(\delta^2),
\end{equation}
which also suggests a second order approximation. 
Thanks to the specific choice of the nonlinearity in~\eqref{eq:dDLSS0}, the discrete solutions share surprisingly many properties with the original PDE~\eqref{eq:DLSS}: solutions to~\eqref{eq:dDLSS0} remain non-negative and of unit mass if the initial datum is so, and we are able to identify discrete counterparts of the two most significant Lyapunov functionals, of suitable contraction estimate, and of two gradient flow structures; see Section \ref{sct:QDDnum} below.

The main result of this paper concerns convergence (see Theorem \ref{thm:limit} for the precise statement):
\begin{manualresult}{A}\label{result:convergence}
	Let $\bar\rho\in \change{\Prob(\crc)}$ be an initial probability \change{measure} on $\crc$. For each $N=2,3,\ldots$, define a piecewise constant approximation $\bar\rho^N$ via $\bar\rho^N_\kappa := \frac1N+\frac{N-1}N\int_{(\kappa-1/2)/N}^{(\kappa+1/2)/N}\bar\rho(x)\dd x$, and let $(\rho^N(t))_{t\ge0}$ be the respective solution to \eqref{eq:dDLSS0} with initial datum $\bar\rho^N$. Then, as $N\to\infty$, the functions $\rho^N$ converge strongly \change{along a subsequence} in $L^2(0,T;L^2(\crc))$ to a weak solution $(\rho_*(t))_{t\ge0}$ of~\eqref{eq:DLSS} with initial datum $\bar\rho$.
\end{manualresult}
\subsection{DLSS equation --- origins and fundamental analytical results}
\label{sct:history}
In its one-dimensional version (actually, on the real half-line), equation \eqref{eq:DLSS} has first appeared in~\cite{DerridaLebowitzSpeerSpohn1991a,DerridaLebowitzSpeerSpohn1991b} as a macroscopic description of interface fluctuations in the anchored Toom model, see also \cite{BordenaveGermainTrodgon} for a more recent contribution on lower-order corrections. The multi-dimensional version of \eqref{eq:DLSS} is known from semi-conductor physics. More precisely, the combination with linear drift and diffusion,
\begin{align}
    \label{eq:DLSSd}
    \partial_t\rho = \Delta \rho + \nabla\cdot(\rho\,\nabla V) - \epsilon^2\nabla\cdot\left(\rho\,\nabla\frac{\Delta\sqrt\rho}{\sqrt\rho}\right),
\end{align}
where the small parameter $\epsilon$ is of the order of Planck's constant $\hbar$, appears as a quantum-corrected version \cite{Ancona} of the classical drift diffusion equation for charge carrier transport in an external or self-consistently coupled potential $V$; note that the fourth order term imposes an additional drift along the gradient of the self-induced Bohm potential. A systematic derivation of \eqref{eq:DLSSd} by the moment method from a quantum kinetic model has been performed in \cite{DegondMehatsRinghofer}. There, \eqref{eq:DLSSd} appears as local approximation of a sophisticated non-local drift-diffusion model, and both the full non-local model~\cite{Pinaud} as well as other higher order local approximations \cite{BukalJuengelMatthes,MatthesRott} have been analyzed. 

The rigorous analysis of \eqref{eq:DLSS} started with a proof of local-in-time well-posedness for positive classical solutions \cite{Bleher}. Later, existence of global-in-time non-negative weak solutions has been obtained in space dimensions $d=1$ \cite{JunPin-analysis}, $d\le3$ \cite{JuengelMatthes2008}, and finally for arbitrary $d\ge1$ \cite{GianazzaSavareToscani2009}. More recently, also uniqueness \change{in a class of suitable regular solutions} \cite{Fischer-uniqueness} and infinite speed of support growth \cite{Fischer-support} have been shown. The self-similar long-time asymptotics towards a Gaussian profile for solutions on $\R^d$ have been formally derived in \cite{CarrilloToscani,CarrilloTang} and were then rigorously proven in \cite{GianazzaSavareToscani2009}, see also \cite{MatthesMcCannSavare} for a generalization. For solutions on bounded domains, exponential convergence to the homogeneous steady state has been shown in \cite{JunPin-analysis,JuengelMatthes2008}.

\subsection{DLSS equation --- structural properties}
\label{sct:QDDprop}
To begin with, the DLSS equation \eqref{eq:DLSS} admits a global weak solution that is non-negative and mass preserving for any non-negative initial datum of finite entropy (see below) \cite{GianazzaSavareToscani2009,JuengelMatthes2008,JunPin-analysis}. By the scaling properties of \eqref{eq:DLSS}, there is no loss in restricting attention to solutions in the space of probability densities from now on.

Next, solutions to \eqref{eq:DLSS} dissipate a variety of Lyapunov functionals. The most essential ones for the structural considerations are the entropy $\ent$ and Fisher information $\fish$, given by
\begin{equation}\label{eq:def:ent-fish}
	\ent(\rho) := \begin{cases}
		\int_\crc \rho(x) \bra*{\log \rho(x)-1} \dd x , & \rho \ll\!  \dd x \\
		+\infty , & \text{else},
	\end{cases}
	\ \ \text{and}\ \ 
	\fish(\rho) := \begin{cases}
		2\int_\crc \abs*{ \partial_x \sqrt{\rho}}^2 \dd x , & \sqrt{\rho}\in H^1(\crc) \\
		+\infty , & \text{else} . 
	\end{cases}
\end{equation}
Actually, $\ent$ and $\fish$ are elements of families of time-monotone power-type functionals 
\begin{align}
    \label{eq:DLSSlyapunov}
    \Ifnc_\alpha(\rho) := \int_\crc\frac{\rho^\alpha}{\alpha(\alpha-1)}\dd x
    \qquad\text{and}\qquad
    \Jfnc_\beta(\rho) := \int_\crc \bigl|\partial_x\big(\rho^{\beta/2}\big)\bigr|^2\dd x,
\end{align}
with parameters $\alpha\in I^{\alpha}$, $\beta\in I^{\beta}$ from suitable intervals $I^{\alpha},\,I^\beta\subset\Rp$, see \cite{JuengelMatthes-entropy}. Specifically, $\fish=\Jfnc_1$, and $\ent=\lim_{\alpha\to1}\Ifnc_\alpha$; for $\alpha\downarrow0$, the limit of $I^\alpha$ yields a further Lyapunov functional 
\begin{align*}
    \fnclog(\rho) := -\int_\crc \log\rho\dd x,
\end{align*}
that is particularly useful for studying positivity of solutions, and has been heavily used in the first existence proof \cite{JunPin-analysis}. 

The more sophisticated structural properties of \eqref{eq:DLSS} concern its behaviour with respect to certain metrics. These are related to different alternative representations of \eqref{eq:DLSS}:
\begin{align}
    \label{eq:dlssVH}
    \partial_t\sqrt\rho &= -\partial_{xxxx}\sqrt\rho+\frac{\big(\partial_{xx}\sqrt\rho\big)^2}{\sqrt\rho}, \\
    \label{eq:dlssVW}
    \partial_t\rho 
    &= - 2 \partial_x \biggl( \rho\, \partial_x \biggl( \frac{\partial_{xx} \sqrt{\rho}}{\sqrt{\rho}}\biggr)\biggr)
    = \partial_{x}\big(\rho\,\partial_{x}\fish'(\rho)\big), \\
    \label{eq:dlssVD}
    \partial_t\rho 
    &= -\partial_{xx}\big(\rho\,\partial_{xx}\log\rho\big) 
    = -\partial_{xx}\big(\rho\,\partial_{xx}\ent'(\rho)\big),
\end{align}
where $\ent'(\rho)=\log\rho$ and $\fish'(\rho)=-2\,\partial_{xx}\sqrt\rho/\sqrt\rho$ denote the variational derivatives of entropy and Fisher information, respectively.
Introduce the Hellinger, the Wasserstein and diffusive transport distances, respectively, between probability measures $\mu_0$ and $\mu_1$ on $\crc$ by their respective dynamical formulations:
\begin{align}
    \label{eq:predef.hellinger}
	\hell(\mu_0,\mu_1)^2 
	&:= \inf_{(\mu_s,u_s)_{s\in[0,1]}}\left\{\int_0^1 \int_\crc u_s^2\dd\mu_s \dd s \,\middle|\,\partial_s\mu_s-u_s\mu_s=0\right\}, \\	
    \label{eq:predef.wasserstein}
    \wass(\mu_0,\mu_1)^2 
	&:= \inf_{(\mu_s,v_s)_{s\in[0,1]}}\left\{\int_0^1 \int_\crc v_s^2\dd\mu_s \dd s \,\middle|\,\partial_s\mu_s + \partial_x(v_s\mu_s)= 0 \right\}, \\
    \label{eq:predef.diffuse}
    \dst(\mu_0,\mu_1)^2 
	&:= \inf_{(\mu_s,w_s)_{s\in[0,1]}}\left\{\int_0^1 \int_\crc w_s^2\dd\mu_s \dd s \,\middle|\,\partial_s\mu_s-\partial_{xx}(w_s\mu_s)=0\right\},
\end{align}
where the infimum in each case is taken over pairs consisting of an $s$-dependent probability measure~$\mu_s$ on $\crc$ and a generalized velocity function $u_s$, $v_s$ or $w_s$, respectively, on $[0,1]\times\crc$, satisfying the given \change{reaction, first and second order continuity equation in distributional sense, respectively}. If $\mu_0=\rho_0\Leb$ and $\mu_1=\rho_1\Leb$ are absolutely continuous, then the dynamical formulation \eqref{eq:predef.hellinger} of the Hellinger distance is easily seen to be equivalent to the more classical definition $\hell(\rho_0,\rho_1) = \|\sqrt{\rho_1}-\sqrt{\rho_0}\|_{L^2}$. The dynamical formulation \eqref{eq:predef.wasserstein} of the Wasserstein distance has been obtained in~\cite{BenamouBrenier2000}. The definition \eqref{eq:predef.diffuse} of the diffusive transport distance is given here for the first time; with the condition $w\ge0$, this becomes a representation of marginal optimal transport \cite{HuesmannTrevisan2019}.

With respect to $\hell$, the flow of \eqref{eq:DLSS} is contractive, i.e.,
\begin{equation*}%
	\hell(\rho_1(t),\rho_2(t)) \leq \hell(\rho_1(s),\rho_2(s)) \qquad\text{for all } 0 \leq s \leq t ,
\end{equation*}
for any two solution $\rho_1, \rho_2$ to~\eqref{eq:DLSS}. Formally, this is an easy consequence of the third representation~\eqref{eq:dlssVH}, as has already been observed in \cite{JunPin-analysis}. A more rigorous treatment of this contractivity, with application to uniqueness can be found in \cite{Fischer-uniqueness}. With respect to $\wass$ and to $\dst$, \eqref{eq:DLSS} is a metric gradient flow \cite{AGS}, for the respective potentials $\fish$ and~$\ent$ from \eqref{eq:def:ent-fish}. Formally, this is seen by matching the dynamical formulations \eqref{eq:predef.wasserstein} and \eqref{eq:predef.diffuse} with the representations \eqref{eq:dlssVW} and \eqref{eq:dlssVD}. The gradient flow formulation with respect to $\wass$ has been made rigorous in \cite{GianazzaSavareToscani2009}. The observation of the gradient flow representation with respect to $\dst$ appears to be novel. 

In Section \ref{sec:DiffTrans:continuous}, we obtain the following:
\begin{manualresult}{B}\label{Result:DiffTrans}
    The diffusive transport $\dst$ defined in~\eqref{eq:predef.diffuse} is a metric $\dst$ on $\Prob(\crc)$ and \eqref{eq:DLSS} is (formally) the metric gradient flow of $\ent$ with respect to $\dst$.
    Moreover, the solutions obtained in Result~\ref{result:convergence} are Hölder continuous in time on $[0,\infty)$ with respect to the diffusive transport distance~$\dst$.
\end{manualresult}
As said before, the dynamic formulation of $\dst$ in~\eqref{eq:predef.diffuse} resembles the Benamou-Brenier formulation of martingale transport introduced in~\cite{HuesmannTrevisan2019}, see also~\cite{BeiglboeckJuillet2016,Backhoff-VeraguasBeiglboeckHuesmannKaellblad2020} and~\cite[§5.1]{Brenier2020hiddenconvexity}. In martingale transport, the diffusive field $(w_s)_{s\in [0,1]}$ is assumed to be non-negative, and thus possesses a stochastic pathwise formulation~\cite{TanTouzi2013,GalichonHenryLabordereTouzi2014} based on duality and stochastic control. 
For our considerations, the bi-directionality of the diffusion in~\eqref{eq:predef.diffuse} is essential, since \eqref{eq:DLSS} is formally “anti-diffusive” in regions where~$\rho$ is not log-concave, see Figures~\ref{fig:diffusive} and~\ref{fig:entropy}.

Recently, fourth order corrections of gradient flow type to the second order heat equation have been suggested in the physics literature~\cite{PanXuLouYao2006,Nika2023}. Also these flows use metrics similar to $\dst$ involving second order derivatives, however, with a mobility $\rho^2$ instead of~$\rho$.
Several examples of linear and non-linear \emph{second order} diffusion equations that are gradient flows with respect to $\dst$ --- including the linear heat and the quadratic porous medium eqaution --- are discussed in \cite{MRS}. Under suitable hypotheses, these flows are even semi-contractive in $\dst$.

\begin{figure}[htbp]\vspace{-0.5\baselineskip}
\begin{minipage}[b]{0.49\linewidth}
	\includegraphics[width=\linewidth]{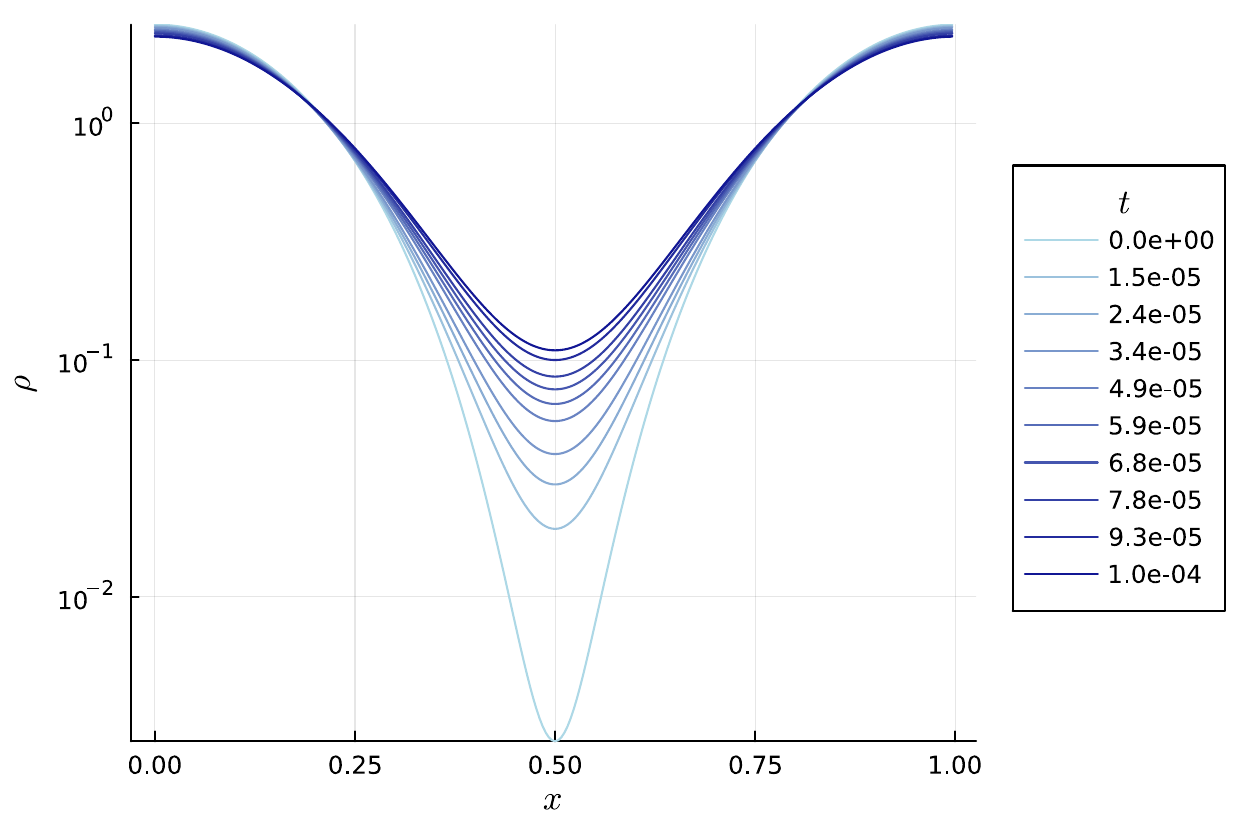}\\
	\includegraphics[width=\linewidth]{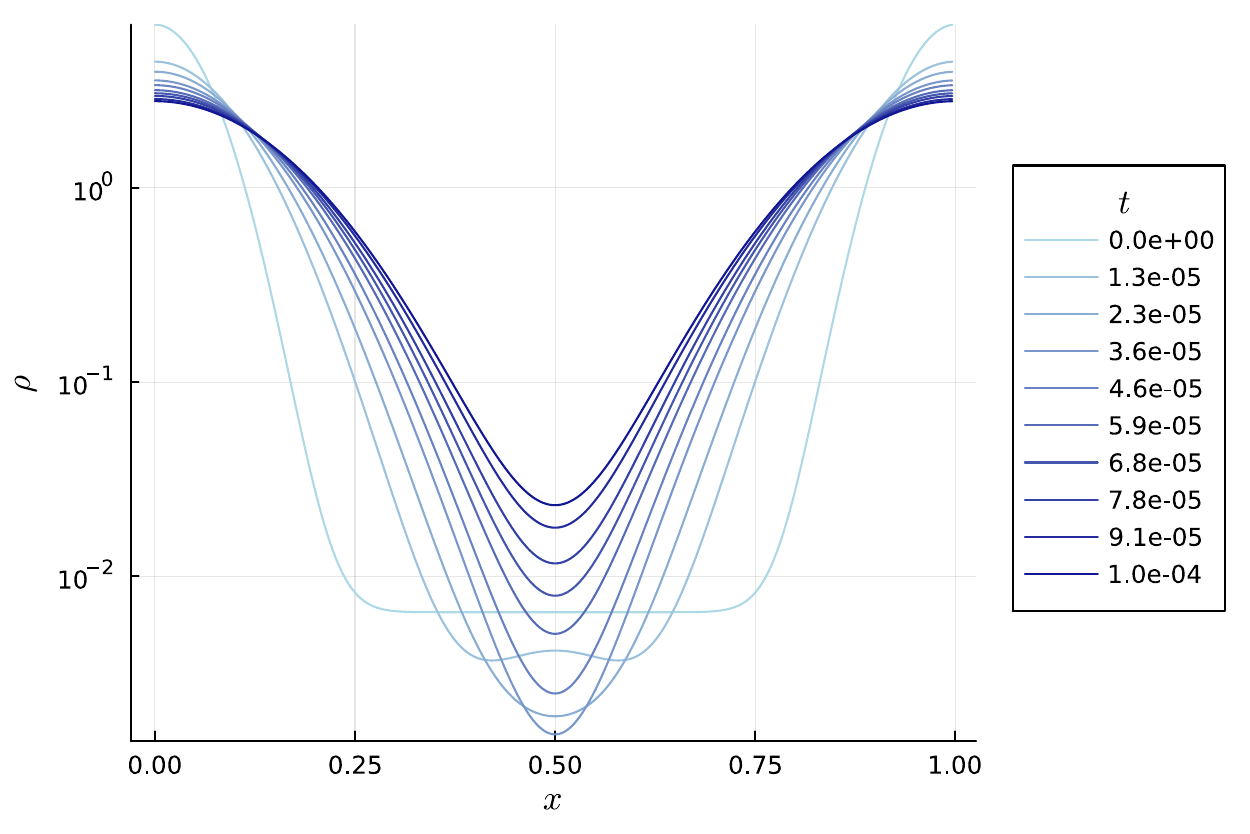}	
\end{minipage}\hfill
\begin{minipage}[b]{0.49\linewidth}
	\includegraphics[width=\linewidth]{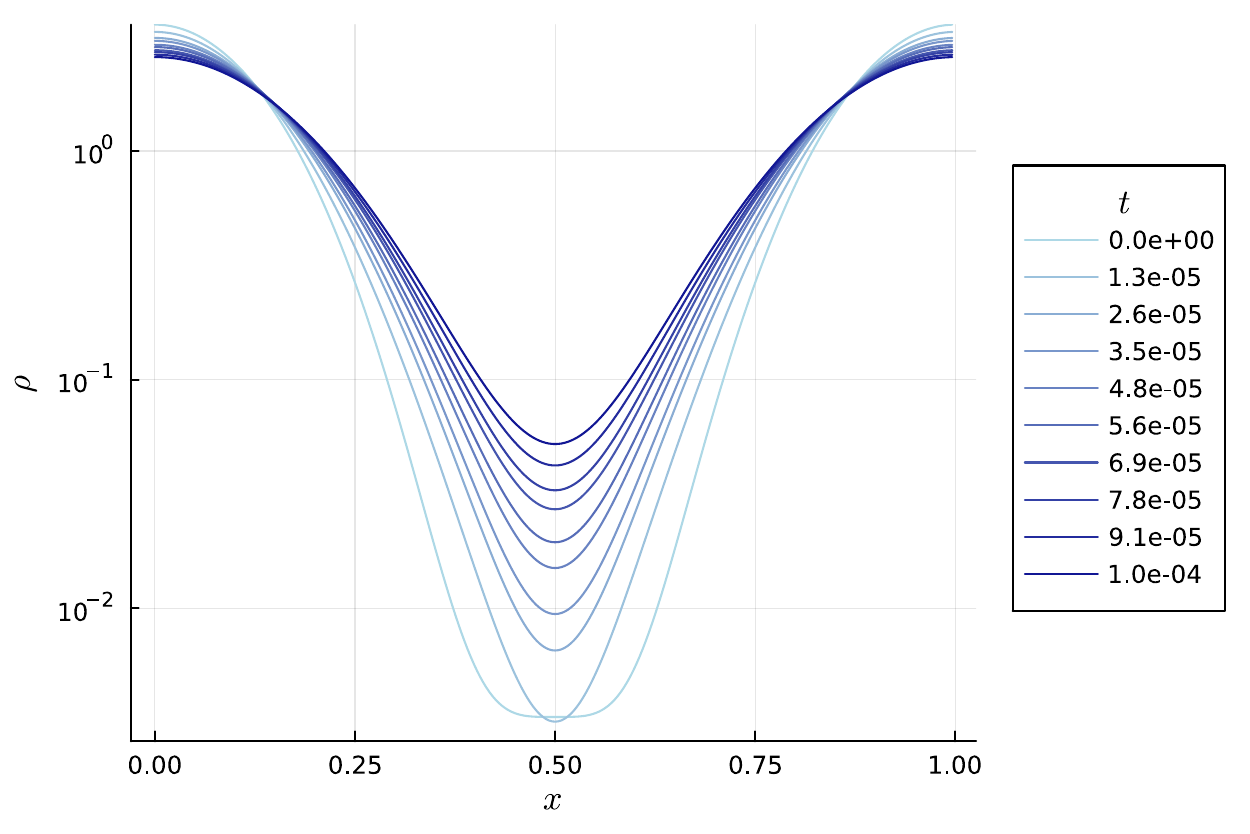}\\
	\includegraphics[width=\linewidth]{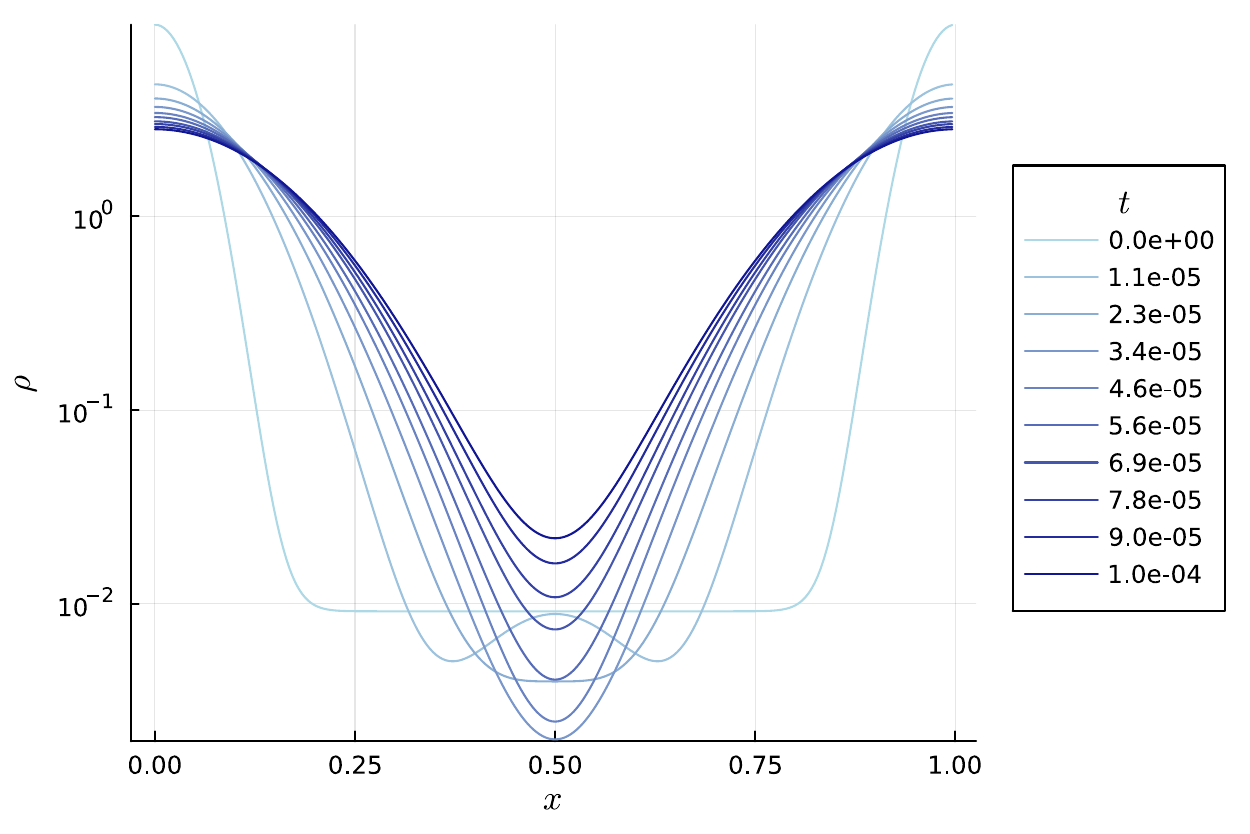}
\end{minipage}
\caption{Logarithmic plot of numerical solution to~\eqref{eq:dDLSS0} started from a discretization of $\bar \rho(x) = \change{Z_{m,\eps}^{-1}} \bra*{\eps^{1/2}+ \bra*{\frac{1+\cos(2\pi x)}{2}}^{\!m}}^{\!2}$ \change{with normalization constant $Z_{m,\eps}$ such that $\bar\rho\in \Prob(\crc)$}, and parameters $\eps=0.001$, $m=1,2,8,16$ (top left, top right, bottom left, bottom right).}	
\label{fig:diffusive}
\end{figure}

\begin{figure}[htbp]\vspace{-0.5\baselineskip}
\centering
\includegraphics[width=0.48\linewidth]{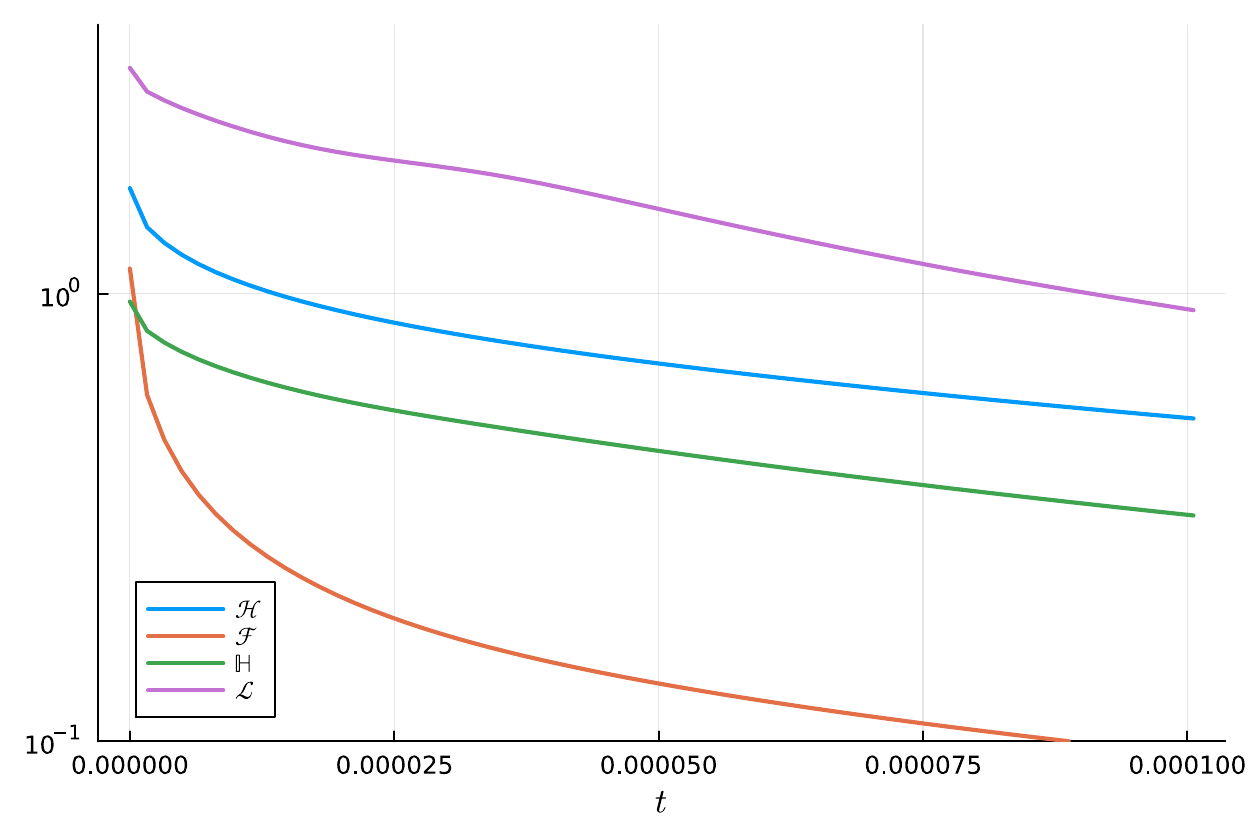}\hfill
\includegraphics[width=0.48\linewidth]{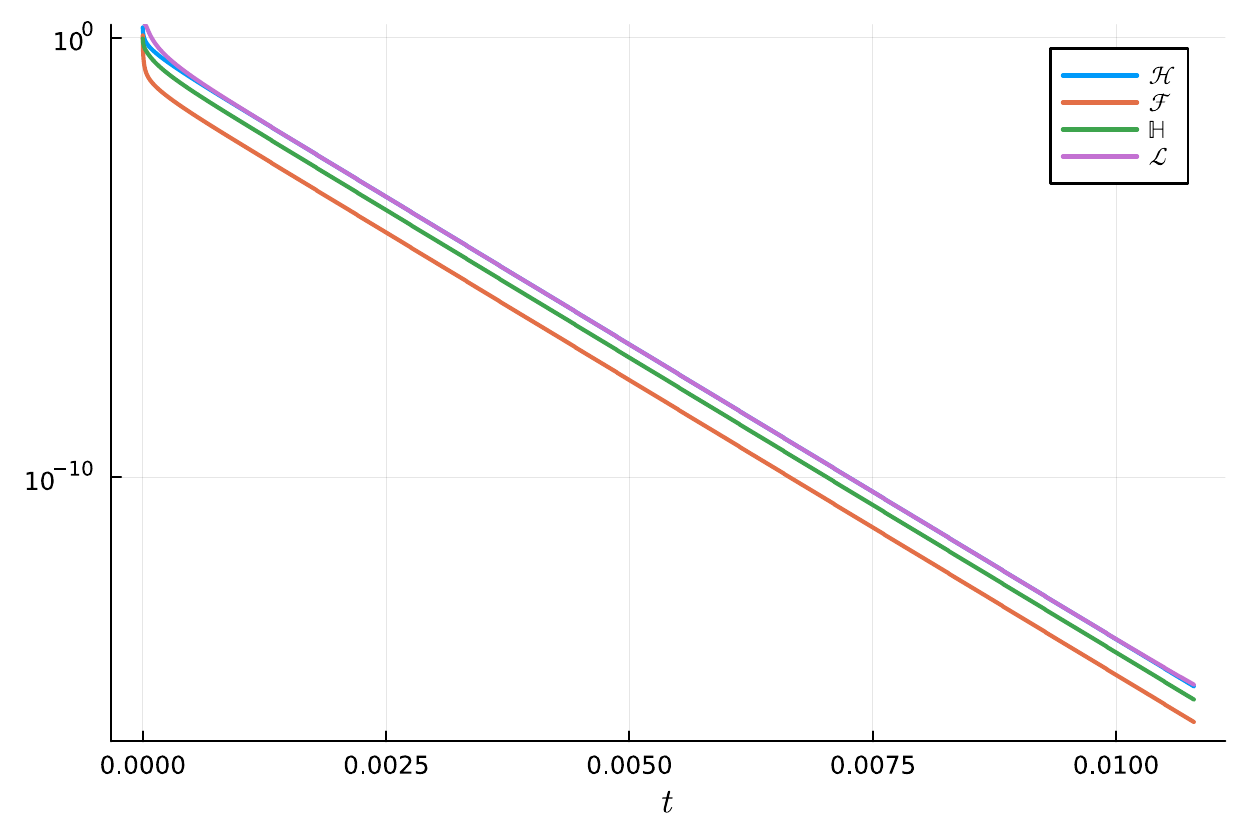}%
\caption{%
Semi-logarithmic plot of discretized Lyapunov functions $\ent$, $\fish$, $\hell(\cdot,1)$, and~$\fnclog$ for the initial datum from Figure~\ref{fig:diffusive} with $m=16$ (left: initial time interval; right: convergence till machine precision with asymptotic exponential rate $-12.2$).}
\label{fig:entropy}
\end{figure}

\subsection{DLSS equation --- numerical schemes}
Various genuinely different approaches to the numerical approximation of \eqref{eq:DLSS} have been proposed in the literature, with different ways to master the central challenge of obtaining non-negative solutions. Essentially all of the available schemes preserve a (typically small) selection of the aforementioned structural properties to the discrete level, and come with analytical results on convergence and/or large-time asymptotics of the discrete solutions.

A first group of schemes \cite{JunPin-numerics, BukalEmmrichJuengel2013, JuengelViolet,JuengelPina,Bukal2021} is concerned with the analysis of a certain semi-discretizations in time; the latter are then typically combined with ad hoc discretizations in space for numerical experiments. The temporal discretizations are designed to preserve specific Lyapunov functionals of the type $\Ifnc_\alpha$, with $\alpha\in I^\alpha$, from \eqref{eq:DLSSlyapunov}, including $\ent$ and $\lyp$. The discretization in \cite{JunPin-numerics} dissipates both $\ent$ and $\lyp$, the first provides convergence, the second positivity. The discretization in~\cite{Bukal2021} preserves the dissipation of $\Ifnc_{1/2}$, and additionally the contractivity in the Hellinger distance. For the schemes in \cite{BukalEmmrichJuengel2013,JuengelViolet}, a parameter can be chosen to select a specific $\alpha\in I^\alpha$ for $\Ifnc_\alpha$ to be dissipated; the extension in \cite{JuengelPina} dissipates $\Ifnc_\alpha$ for all $\alpha\in I^\alpha$ simultaneously. In each case, the respective dissipations provide an $H^2$-control on either $\log\rho$, or some root $\sqrt[m]\rho$, which are the sources for the convergence analysis. In \cite{JunPin-numerics2}, an generalization of \cite{JunPin-numerics} to multiple dimensions is proposed, where \eqref{eq:DLSS} is augmented by an additional term for lattice temperature that enforces positivity.

	A second group \cite{CarrilloTang,MaasMatthes,LiLuWang} is formed by schemes that are fully discrete, with a finite difference discretization in space.  \cite{CarrilloTang} directly builds on the PDE \eqref{eq:DLSS}, while \cite{LiLuWang,MaasMatthes} starts from the formulation as Wasserstein gradient flow. In each case, it is proven that the respective scheme admits non-negative solutions that dissipate a certain discretized version of $\ent$ or $\fish$, and, in case of \cite{MaasMatthes}, reproduce the long-time asymptotics including rates. There is no convergence analysis available. 

A third group of schemes \cite{DueringMatthesMilisic,MatthesOsberger} is fully Lagrangian. These schemes rely on the isometry between the $L^2$-Wasserstein metric and the $L^2$-norm in one space dimension. They are variational and dissipate a discretized Fisher information $\fish$. In \cite{MatthesOsberger}, additi	onally a version of $\ent$ is dissipated, which forms the basis for the convergence proof.

About our own discretization \eqref{eq:dDLSS0}, we prove in Section \ref{sct:dQDDprop} the following properties:
\begin{manualresult}{C}\label{Result:Scheme}
    The discretization \eqref{eq:dDLSS0} shares the following properties with \eqref{eq:DLSS}: existence of non-negative, mass-preserving solutions; dissipation of (discretized versions of) $\ent$, $\fish$ and $\fnclog$; contractivity in the Hellinger distance; gradient flow structure with respect to a (discretized) $L^2$-Wasserstein metric; the gradient flow structure with respect to a (discretized) diffusive transport metric; and an additional generalized gradient flow structure.
\end{manualresult}
Figure~\ref{fig:entropy} visualizes the decay of the different Lyapunov function.
We implemented the scheme~\eqref{eq:dDLSS0} using an implicit Euler approximation for the time derivative in the Julia language~\cite{Julia-2017}. 
The fixed point problem for the implicit time-stepping is solved using a Newton method through the NLSolve-package~\cite{NLSolve}. 
\change{
In our numerical test, we use an adaptive time-stepping algorithm which at the same time resolves the busy initial phase of the dynamics and allows to study the long-time asymptotics in a single simulation.
The adaptation is implemented the follows: the initial time-step is $\delta t=\delta/100$; later, $\delta t$ halved if more than four Newton iterations are needed, and is increased by five percent if less than three Newton iterations are sufficient. 
}
The simulation is stopped when the entropy is of the order of the machine precision.

\change{We haven chosen the implicit Euler approximation since this preserves the monotonicity of convex Lyapunov functionals. Specifically here, (discretized versions of) the entropy $\ent$, the Fisher information $\fish$, and the Lyapunov functional $\lyp$ are dissipated; monotonicity of $\lyp$ then further guarantees positivity of the solutions. For the same reason, the total mass is preserved, and the time stepping is contractive in the Hellinger distance.
We remark that the use of a \emph{variational} implicit time stepping --- in the spirit of de Giorgi's minimizing movements and in analogy to the JKO scheme~\cite{JKO1998} for Wasserstein gradient flows --- might appear more natural than the implicit Euler method. However, a genuinely variational approach will require a more thorough understanding of the underlying metric~$\dst$, see~\eqref{eq:predef.diffuse}, and an efficient method to compute its spatial discretization. We leave this for future research.}

\section{Diffusive transport on the continuous torus}\label{sec:DiffTrans:continuous}
In this section, we make the formal structure \eqref{eq:predef.diffuse} rigorous along the lines of \cite{DNS}, see also \cite{CLSS}. Below (and only here), we distinguish between probability measures $\mu\in\Prob(\crc)$ and their Lebesgue densities $\rho\in L^1(\crc)$. Further, denote by $\radon(\crc)$ the space of Radon measures on $\crc$. On a pair~$(\mu,\moment)$ with $\mu\in\Prob(\crc)$ and $\moment\in\radon(\crc)$, we define the \emph{action density} by
\begin{align}\label{eq:def:action_density}
	\act(\mu,\moment) := 
	\begin{cases}
	\displaystyle	\int_\crc \biggl|\frac{\dn\moment}{\dn\mu}\biggr|^2\dd\mu, & \moment\ll \mu;\\
		+\infty , & \text{else}.
	\end{cases}
\end{align}
If $\mu$ is absolutely continuous with density function $\rho$ with respect to Lebesgue, and $\moment$ is absolutely continuous with density function $w$ with respect to $\mu$, then 
\begin{align}
    \label{eq:reducedactionfunction}
    \act(\rho\Leb,w\,\rho\Leb) = \int_\crc \rho\, w^2\dd x.
\end{align}
The functional $\act$ is lower semi-continuous with respect to vague convergence --- this follows in analogy to \cite[Lemma 3.9]{DNS}.

Next, we consider parametrized families $(\mu^s,\moment^s)_{s\in[0,1]}$ of pairs $\mu^s\in\Prob(\crc)$, $\moment^s\in\radon(\crc)$, measurable with respect to $s$. Such a family is called a \emph{curve} if it satisfies the second-order continuity equation
\begin{align}\label{eq:diffusive:CE}
	\partial_s\mu^s = \partial_{xx}\moment^s
\end{align}
in the sense of distributions on $\crc\times[0,1]$. Moreover, a curve is \emph{of finite action}, if
\begin{align}\label{eq:finite:action}
	\action\big[(\mu^s,\moment^s)_{s\in[0,1]}\big] := \int_0^1 \act(\mu^s,\moment^s)\dd s < \infty .
\end{align}
An adaptation of the proof of \cite[Lemma 8.1.2]{AGS} easily shows that for a curve $(\mu^s,\moment^s)_{s\in[0,1]}$ of finite action, one may always assume (after modifications on a set of \change{zero} measure in $[0,1]$) that $s\mapsto\mu^s$ is weakly continuous, and in particular has well-defined initial and terminal values $\mu^0$ and $\mu^1$, respectively. Accordingly, for given $\mu^0,\mu^1\in\Prob(\crc)$, the set $\connect(\mu^0,\mu^1)$ of curves $(\mu^s,\moment^s)_{s\in[0,1]}$ of finite action~\eqref{eq:finite:action} satisfying~\eqref{eq:diffusive:CE} that connect $\mu^0$ to $\mu^1$ is a well-defined object.

In the following, we use \change{the homogeneous Sobolev space $\dot H^{2}(\crc)$ defined by 
\begin{equation*}
	\dot H^2(\crc) := \set*{ f\in H^2(\crc): \int_\crc f \dx{x} = 0} \,,
\end{equation*}
which is a Banach space with respect to the norm $\|f\|_{\dot H^2(\crc)}:=\|\partial_{xx}f\|_{L^2(\crc)}$
thanks to the Poincaré inequality on $\crc$. The dual Sobolev space $\dot H^{-2}(\crc)$ carries the norm 
\begin{equation*}%
	\norm{\nu}_{\dot H^{-2}(\crc)} := \sup\set*{ \abs*{\skp{ f, \nu}} : f\in \dot H^2(\Omega),\,\norm{f}_{\dot H^2(\crc)} \leq 1 } . 
\end{equation*}
Since $\dot H^2(\crc)$ embeds continuously in the continuous functions $C(\crc)$, its dual $\dot H^{-2}(\crc)$ contains in particular the Radon measures of zero average on $\crc$. In the proof of Lemma \ref{lem:H-2} below, we use an equivalent formulation of the norm in $\dot H^{-2}(\crc)$ on the difference of two probability measures $\mu^0,\mu^1\in\Prob(\crc)$:
\begin{equation*}%
	\norm{\mu^1-\mu^0}_{\dot H^{-2}(\crc)}^2 = \inf\set*{ \int_{\crc} \abs{w}^2 \dd x : \partial_{xx} w = \mu^1-\mu^0 \text{ in distributional sense} } .
\end{equation*}
Note that the double primitive $w$ is uniquely determined up to a constant on $\crc$. 
It is is easy to see that the infinimum is attained for the one of zero average,
\begin{equation}\label{eq:def:H-2:doubprim}
	\norm{\mu^1-\mu^0}_{\dot H^{-2}(\crc)}^2 = \int_{\crc} \abs{W}^2 \dd x \ \text{ with }\ \partial_{xx} W = \mu^1-\mu^0 \text{ in distributional sense and } \int_\crc W \dd x =0 .
\end{equation}}
\begin{lemma}[Comparison with $\dot H^{-2}(\crc)$-norm]\label{lem:H-2}
	For any given $\mu^0,\mu^1\in\Prob(\crc)$, 
    there exists a connecting curve $(\mu^s,\moment^s)_{s\in[0,1]}\in\connect(\mu^0,\mu^1)$ with  
	\begin{equation}\label{eq:comp:H-2}
		 \action\big[(\mu^s,\moment^s)_{s\in[0,1]}\big] 
        \le 2 \bigl(-\log \|\mu^1-\mu^0\|_{\dot H^{-2}(\crc)}^2\bigr)^{-1}.
	\end{equation}
	In particular, any two measures in $\Prob(\crc)$ can be connected by a curve of action less than one.
\end{lemma}
\begin{proof}
	We construct a particular connecting curve of finite action for any $\tau\in (0,1/2)$ as follows. For $0<s\le\tau$, let $\mu^s_\tau$ be the time-$s$-solution to the heat equation $\partial_s\mu^s_\tau =\partial_{xx}\mu^s_\tau$ with initial datum $\mu_\tau^0=\mu^0$, and for $1-\tau\le s<1$ let $\mu^s_\tau$ by the time-$s$-solution to the time-reversed heat equation $\partial_s\mu^s_\tau =-\partial_{xx}\mu^s_\tau$ with terminal datum $\mu_\tau^1=\mu^1$. For $s\in(\tau,1-\tau)$, define $\mu_\tau^s$ by linear interpolation,
	\[ \mu^s_\tau := \Bigl(\frac{1-\tau-s}{1-2\tau}\Bigr)\mu^{\tau}_\tau+\Bigl(\frac{s-\tau}{1-2\tau}\Bigr)\mu^{1-\tau}_\tau.\]
    By the properties of the heat equation, $\mu_\tau^s$ is absolutely continuous with respect to the Lebesgue measure on $\crc$ at each $s\in(0,1)$; denote the corresponding probability density by $\rho_\tau^s$. We obtain a positive lower bound on $\rho_\tau^s$ for $s \le\tau$ and for $s\ge1-\tau$ using the representation by means of convolution with the  periodic heat kernel $k^s$, see~\cite{Ohyama1995}, given by 
	\begin{align*}
		k^s(z) := \frac{1}{\sqrt{4\pi s}} \sum_{\ell \in \Z} e^{-\frac{(x-\ell)^2}{4 s}} 
		\geq \frac{1}{\sqrt{4\pi s}} \exp\bra*{-\frac{\operatorname{dist}(x,\Z)^2}{4s}} 
        \geq \frac{1}{\sqrt{4\pi s}} \exp\bra*{-\frac{1}{16 s}} . 
	\end{align*}
    In particular, also using that linear interpolation preserves lower bounds,
    \begin{align*}
        \rho_\tau^s \ge \frac1{\sqrt{4\pi\tau}}e^{-1/16\tau} 
        \qquad\text{for $s\in[\tau,1-\tau]$}.
    \end{align*} 
    The continuity equation $\partial_s\mu^s_\tau=\partial_{xx}\moment^s_\tau$ is satisfied with $\moment^s_\tau=\mu^s_\tau$ and $\moment^{s}_\tau=-\mu^{s}_\tau$, respectively, for $s\in[0,\tau]$ and for $s\in[1-\tau,1]$. For $s\in(\tau,1-\tau)$, we choose $\moment^s_\tau=w_\tau$ independently of $s$, where $w_\tau=(1-2\tau)^{-1}W$, and $W$ is the mean-zero double primitive of $\rho_\tau^{1-\tau}-\rho_\tau^\tau$, which exists since $\rho_\tau^{1-\tau}$ and $\rho_\tau^\tau$ have the same (unit) mass. By definition of the $\dot H^{-2}$-norm as in~\eqref{eq:def:H-2:doubprim}, and since the heat flow is non-expansive on $\dot H^{-2}(\crc)$,
    \begin{align*}
        \|W\|_{L^2(\crc)}^2 = \norm{\mu_\tau^{1-\tau}-\mu_\tau^\tau}_{\dot H^{-2}(\crc)}^2 \le H:= \norm{\mu^1-\mu^0}_{\dot H^{-2}(\crc)}^2. 
    \end{align*}
    A double primitive $w$ of $\mu^0-\mu^1$ is given at any $x\in [0,1]\mathrel{\hat=}\crc$ by 
    	\[
    	  W(x) := \int_{1/2}^x \int_{1/2}^y \dd(\mu^0-\mu^1)(z) \dd y. 
    	\]
    Since $\mu^0,\mu^1\in\Prob(\crc)$, we observe that $\abs{W(x)} \leq \abs{x-1/2}$, hence from \eqref{eq:def:H-2:doubprim}, we conclude that
    \begin{align}\label{eq:H:upper}
        H \le \norm{W}_{L^2(\crc)}^2 \leq \int_{0}^1 \abs{x-1/2}^2 \dd x= \frac{1}{12}.
	\end{align}
    The action can now be estimated as follows:
    \begin{align*}
		\action\big[(\mu^s_\tau,\moment^s_\tau)_{s\in[0,1]}\big]
		&\leq \int_0^{\tau} \int_{\crc} \frac{(\rho_\tau^s)^2}{\rho_\tau^s}\dd x\dd s + \int_{\tau}^{1-\tau} \frac{w_\tau^2}{\rho_\tau^s}\dd x\dd s + \int_{1-\tau}^1 \int_{\crc} \frac{(\rho_\tau^s)^2}{\rho_\tau^s}\dd x\dd s \\
        &\le \int_0^\tau 1 \dd s + \frac{\sqrt{4\pi\tau}e^{1/(16\tau)}}{(1-2\tau)^2}\int_\tau^{1-\tau} \|W\|_{L^2(\crc)}^2\dd s + \int_{1-\tau}^1 1 \dd s\\
        &\le 2\tau + \frac{\sqrt{4\pi \tau} e^{1/(16\tau)}}{1-2\tau} H. 
	\end{align*}
    We can choose $\tau$ as
    \begin{align*}
        \tau:=\frac1{8} (-\log H)^{-1} \in (0,1/16),
    \end{align*}
    so that $e^{1/(16\tau)}=\sqrt H$,
    where the bound follows from~\eqref{eq:H:upper}. With that, we obtain
 \begin{align*}
 	\action\big[(\mu^s_\tau,\moment^s_\tau)_{s\in[0,1]}\big]
 	&\le \frac1{4(-\log H)} + \frac{\sqrt{4\pi/16}\sqrt{H}}{(1-2/16)}\\
 	&\le \frac{1}{-\log H} \bra*{ \frac1{4} + \frac{\sqrt{\pi}}{2} \frac{8}{7} \sqrt{H}(-\log H)} \le 2(-\log H)^{-1},
 \end{align*}
 where for the last step, we use that $H\mapsto \sqrt{H}(-\log H)$ is monotone on $[0,1/12]$ (maximum at $H=e^{-2}\leq  1/12$) and hence bounded by $\sqrt{1/12}\log 12$. The whole term in brackets evaluates to~$1.17$, 
 which gives the claim.
\end{proof}
We can now state the first part of Result~\ref{Result:DiffTrans}.
\begin{proposition}%
	The \emph{diffusive transport distance} $\dst$, defined between given $\mu^0,\mu^1\in\Prob(\crc)$ by
	\begin{align*}
		\dst(\mu^0,\mu^1)
		:= \inf_{(\mu^s,\moment^s)_{s\in[0,1]}\in\connect(\mu^0,\mu^1)}\action\big[(\mu^s,\moment^s)_{s\in[0,1]}\big]^{1/2},
	\end{align*}
	is a metric on $\Prob(\crc)$, and turns $\Prob(\crc)$ into a complete geodesic space metrizing weak convergence.
    Moreover, $\dst$ has the comparison
    \begin{equation}\label{eq:D:comp}
        \norm{\mu^0-\mu^1}_{\dot W^{2,\infty}(\crc)'}^2 \leq \dst(\mu^0,\mu^1)^2 \leq 
        2 \bigl(-\log \|\mu^1-\mu^0\|_{\dot H^{-2}(\crc)}^2\bigr)^{-1} ,
    \end{equation}    
    where $\dot W^{2,\infty}(\crc)'$ denotes the dual norm to $\dot W^{2,\infty}(\crc)$. 
\end{proposition}
\begin{proof}
	From Lemma \ref{lem:H-2} and the fact that $\act$ is non-negative, it follows that $\dst$ is a well-defined map from pairs of probability measures into the non-negative real numbers.%
	
	As an intermediate result, we show that curves $(\mu^s,\moment^s)_{s\in[0,1]}$ of \change{uniformly bounded} action are uniformly H\"older continuous in the dual of \change{$\dot W^{2,\infty}(\crc)$. Let $\varphi\in \dot W^{2,\infty}(\crc)$} be given,
	then by \change{recalling definition~\eqref{eq:def:action_density} of the action density, we can estimate for $0\le t_0<t_1\le1$ by the Hölder inequality}
	\begin{align}
		\int_\crc \varphi\dd(\mu^{t_0}-\mu^{t_1}) 
		&%
		= \int_{t_0}^{t_1}\int_\crc \partial_{xx}\varphi\dd\moment^s\dd s 
        = \int_{t_0}^{t_1}\partial_{xx}\varphi\frac{\dn\moment^s}{\dn\mu^s}\dd\mu^s\dd s \nonumber\\
		&\le \|\varphi\|_{\change{\dot W^{2,\infty}(\crc)}}\biggl(\int_{t_0}^{t_1}\int_\crc\dd\mu^s\dd t\biggr)^{\!1/2} \change{\biggl(\int_{t_0}^{t_1}	\act(\mu_s,\moment_s)\dd t\biggr)^{\!1/2} } \nonumber \\
		&\le \|\varphi\|_{\change{\dot W^{2,\infty}(\crc)}} \change{\action\bigl[(\mu^s,\moment^s)_{s\in[0,1]}\bigr]^{1/2}} \,.\label{eq:dst-cont}
	\end{align}
	To prove from here that the infimum is actually a minimum: consider a minimizing sequence $(\mu^s_n,\moment^s_n)_{s\in[0,1]}$ in $\connect(\mu^0,\mu^1)$. By the $n$-uniform H\"older continuity of the $\mu^s_n$ in the dual of $C^2(\crc)$, the Arzela-Ascoli theorem guarantees the existence of a subsequence converging locally uniformly with respect to $s\in[0,1]$ to a continuous limit $\mu_*^s$. In particular, $(\mu_n^s)_{s\in[0,1]}\to(\mu_*^s)_{s\in[0,1]}$ vaguely on $(0,1)\times\crc$, and initial and terminal value are inherited by $\mu_*$. Further, since the action is uniformly bounded, also the total variation of the $(\moment_n^s)_{s\in[0,1]}$ on $(0,1)\times\crc$ is $n$-uniformly bounded, and for a further subsequence, $(\moment_n^s)_{s\in[0,1]}$ converges vaguely to a limit $(\moment_*^s)_{s\in[0,1]}$. The continuity equation clearly passes to the limit, hence $(\mu_*^s,\moment_*^s)_{s\in[0,1]}\in\connect(\mu^0,\mu^1)$. Finally, by lower semi-continuity of $\action$,
	\begin{align*}
		\dst(\mu^0,\mu^1)^2 
		\le \action\bigl[(\mu_*^s,\moment_*^s)_{s\in[0,1]}\bigr]
		\le \liminf_{n\to\infty}\action\big[(\mu_n^s,\moment_n^s)_{s\in[0,1]}\big]
		=\dst(\mu^0,\mu^1)^2 .
	\end{align*}        
	Symmetry and triangle inequality follow by abstract arguments easily from the definition. It remains to verify the axiom that $\dst(\mu^0,\mu^1)=0$ implies $\mu^0=\mu^1$; but this is a consequence of the existence of a minimizer above.
	\change{Indeed, $\action\bigl[(\mu_*^s,\moment_*^s)_{s\in[0,1]}\bigr]=0$ implies by means of \eqref{eq:dst-cont} above that $\mu_*^s$ is actually independent of $t\in[0,1]$, and in particular that $\mu^0=\mu^1$.}
	In conclusion, we have proven that $\Prob(\crc)$ with $\dst$ is a geodesic metric space.
	
    Next, we verify completeness. To begin with, observe that if $\mu^n\rightharpoonup \mu$ weakly in measures, then also \change{$\norm{\mu^n - \mu}_{\dot H^{-2}(\crc)} \to 0$}, since $\dot H^2(\crc)\hookrightarrow C(\crc)$, and so also $\dst(\mu^n,\mu)\to 0$ by Lemma~\ref{lem:H-2}. Now consider a Cauchy sequence $\mu^n$ with respect to $\dst$; it suffices to show that $\mu^n$ has a weak limit. Actually, by compactness of $\crc$, it is clear that any subsequence of $(\mu^n)$ possesses a weakly convergent sub-subsequence; it thus suffices to show that the respective weak limits are actually independent of the chosen subsequence. This is done as follows: thanks to estimate~\eqref{eq:dst-cont}, we get for any $n,m$ and any $\varphi\in C^2(\crc)$ the bound
	\[
	  \abs*{\int_\crc \varphi \dd (\mu^n - \mu^m)} \leq \norm{\varphi}_{C^2(\crc)} \dst(\mu^n,\mu^m)^{1/2} . 
	\]
	We conclude that also $\int_\crc \varphi \dd (\mu^n - \mu^m)\to 0$ as $n,m\to \infty$ for any $\varphi \in C(\crc)$ by a density argument: take $\norm{\varphi^\eps - \varphi}_{C(\crc)}\leq \eps$ such that $\varphi^\eps\in C^2(\crc)$. Then, for any $\eps>0$ there is $N=N(\eps,\varphi)$ such that for all $n,m\geq N$  
	\begin{align*}
		\abs*{\int_\crc \varphi \dd (\mu^n - \mu^m)} &\leq 2 \eps + \abs*{\int_\crc \varphi^\eps \dd (\mu^n - \mu^m)} \leq 2\eps + \norm{\varphi^\eps}_{C^2(\crc)} \dst(\mu^n,\mu^m)^{1/2} \leq 3\eps .
	\end{align*}
	This shows the desired uniqueness of the limit $\mu$. Hence $\dst$ is complete in $\Prob(\crc)$. The argument further shows that $\dst$ metrizes the weak convergence of measures.

    The upper bound in \eqref{eq:D:comp} is a consequence of~\eqref{eq:comp:H-2} in Lemma~\ref{lem:H-2}. For the lower bound, we take any $\varphi \in \dot W^{2,\infty}(\crc)$ and estimate in analogy to \eqref{eq:dst-cont} above:
    \[
      \int_{\crc} \varphi \dd (\mu^0-\mu^1) = \int_0^1 \int_{\crc} \partial_{xx} \varphi \dd \moment^s \leq \norm{\varphi}_{\dot W^{2,\infty}(\crc)}
      \action\bigl[(\mu^s,\moment^s)_{s\in[0,1]}\bigr]^{\frac{1}{2}} .
    \]
    The desired estimate follows by taking the $\sup$ in $\varphi$ with norm bounded by $1$ and infimizing among $(\mu^s,\moment^s)_{s\in[0,1]}$.
\end{proof}
\begin{lemma}\label{lem:AubinLions}
	 If $\rho\in \AC([0,T],(\Prob(\crc),\dst))$ with $C_\fish:=\sup_{t\in [0,T]} \fish(\rho_t)<\infty$, then already $\rho\Leb\in \AC([0,T],(\Prob(\crc),\hell))$, and 
		\begin{equation*}%
		      \hell(\rho^s,\rho^t) \leq C \dst(\rho^s,\rho^t)^{\frac{1}{12}} \qquad\text{for all $0\leq s \leq t \leq T$},
		\end{equation*}
        with a modulus $C$ of continuity that is expressible in terms of $C_\fish$ alone.
\end{lemma}
\begin{proof}
    Observe that, for any numbers $a,b\ge0$,
    \begin{align}
        \label{eq:elementary}
        \big|\sqrt a-\sqrt b\big|^3 
        = \big|\sqrt a-\sqrt b\big| \big(\sqrt a-\sqrt b\big)^2 
        \le \big(\sqrt a+\sqrt b\big)\big(\sqrt a-\sqrt b\big)\big(\sqrt a-\sqrt b\big)
        = (a-b)\big(\sqrt a-\sqrt b\big).
    \end{align}
	We apply this estimate to $a=\rho_s(x)$ and $b=\rho_t(x)$, combined with Jensen's inequality and $\dot H^1$-$\dot H^{-1}$-interpolation: 
     \begin{equation}\label{eq:HellHoelder}
     \begin{split}         
        \hell(\rho_s,\rho_t)^2 
        &\leq \pra*{ \int_\crc \big| \sqrt{\rho_s} - \sqrt{\rho_t}\big|^3 \dd x }^{\frac{2}{3}} 
         = \pra*{\int_\crc (\rho_s-\rho_t)\big(\sqrt{\rho_s}-\sqrt{\rho_t}\big)}^{\frac23}\\
        &\le \big\|\sqrt{\rho_s}-\sqrt{\rho_t}\big\|_{\dot H^1}^{\frac23}\,\|\rho_s-\rho_t\|_{\dot H^{-1}}^{\frac23}.
     \end{split}
    \end{equation}
    Since
    \begin{align*}
        \big\|\sqrt{\rho_s}-\sqrt{\rho_t}\big\|_{\dot H^1}^2
        \le \int_{\crc}\big(\partial_x\sqrt{\rho_s}-\partial_x\sqrt{\rho_t}\big)^2\dd x
        \le 2\big(\fish(\rho_s)+\fish(\rho_t)\big) \le 2C_\fish,
    \end{align*}
    it is sufficient to bound the distance of $\rho_s$ to $\rho_t$ in $\dot H^{-1}(\crc)$ by a certain power of their distance in~$\dst$. In fact, by the comparison estimate~\eqref{eq:D:comp}, it suffices to show that there is a constant $B$ such that
     \begin{equation}\label{eq:est:H-1:W2'}
      \change{ \norm{ \mu^0-\mu^1}_{\dot H^{-1}(\crc)} \leq B\norm{ \mu^0-\mu^1}_{\dot W^{2,\infty}(\crc)'}^{\frac{1}{4}}} \qquad \change{\text{for any $\mu^0,\mu^1\in \Prob(\crc)$}} \,.
     \end{equation}
     To show~\eqref{eq:est:H-1:W2'}, \change{let for brevity $\nu= \mu^0-\mu^1$} and consider for $\varphi\in \dot H^1(\crc)$ a regularization $\varphi^\eps \in \dot W^{2,\infty}(\crc)$. Then, \change{the triangle and duality inequality give}
     \begin{align}\label{eq:H1:W2infty}
         \abs*{\int_\crc \varphi \dd \nu } \leq \change{ \int_\crc \abs*{\varphi-\varphi_\eps} \dd \abs{\nu} + \abs*{\int_\crc \varphi_\eps \dx{\nu}} } \leq 2\norm{\varphi-\varphi_\eps}_{L^\infty(\crc)} +  \norm{\varphi_\eps}_{\dot W^{2,\infty}(\crc)}\norm{\nu}_{\dot W^{2,\infty}(\crc)'} \,,
     \end{align}
     \change{where we use that $\int \dd \abs{\nu}\leq 2$ as a difference of two probability measures.}
     We assume a regularization of the form $\varphi_\eps = K_\eps \ast \varphi$ with non-negative kernels $K_\eps \in C^\infty(\crc,\Rnn)$ specified below. To estimate the first term on the right-hand side of \eqref{eq:H1:W2infty}, observe that for any $x\in \crc$,
     \begin{align*}
     	\varphi(x)- \bra*{K_\eps \ast \varphi}(x) &= \int_\crc K_\eps(x-y) \bra*{ \varphi(x)-\varphi(y)} \dd y \leq \int_\crc K_\eps(x-y) \int_x^y \partial_z \varphi(z)\dd z \dd y \\
     	&\leq \norm{\varphi}_{\dot H^1(\crc)} \int_\crc K_\eps(x-y) \abs{x-y}_{\crc}^{\frac{1}{2}} \dd y .
     \end{align*}
     For regularization, we use the Van-Mises distribution $K_\eps(x) = c_\eps^{-1} \exp(\eps^{-1} \cos(2\pi x))$, where $c_\eps = \sqrt{2\pi \eps}+ O(\eps)$ for $\eps \ll 1$. A simple scaling argument provides the estimate
     	\begin{align*}
     		 \int_\crc K_\eps(x-y) \abs{x-y}_{\crc}^{\frac{1}{2}} \dd y \le B' \eps^{\frac{1}{4}} ,
     	\end{align*}
     with a constant $B'$ independent of $\eps\in (0,1)$.
     Likewise, we can estimate
     \begin{align*}
     	\norm{\varphi_\eps}_{\dot W^{2,\infty}(\crc)} = \sup_{x\in \crc} \abs*{\int_\crc \partial_x K_\eps(x-y)\partial_y\varphi(y)\dd y} \leq \norm{K_\eps}_{\dot H^1(\crc)}\norm{\varphi}_{\dot H^1(\crc)} 
     \end{align*}
     and another scaling argument shows that
     \begin{align*}
     	 \norm{K_\eps}_{\dot H^1(\crc)} \le B'' \eps^{-\frac{3}{4}} .
     \end{align*}
     In total, we conclude from~\eqref{eq:H1:W2infty} the estimate
     \begin{equation*}
     	\abs[\bigg]{\int_\crc \varphi \dd \nu } \le B''' \bra*{\eps^{\frac{1}{4}} + \eps^{-\frac{3}{4}}\norm{\nu}_{\dot W^{2,\infty}(\crc)'}} \norm{\varphi}_{\dot H^1(\crc)} .
     \end{equation*}
     Choosing $\eps=\norm{\nu}_{\dot W^{2,\infty}(\crc)'}$ and taking the supremum in $\varphi$ with $\norm{\varphi}_{\dot H^1(\crc)}\leq 1$ yields the desired bound~\eqref{eq:est:H-1:W2'}.
\end{proof}

\subsection{Gradient flows with respect to the diffusive transport distance}
For the identification of \eqref{eq:DLSS} as gradient flow of $\ent$ in the novel metric $\dst$ as stated in Result~\ref{Result:DiffTrans}, we proceed formally. A fully rigorous analysis, including an appropriate definition of $\ent$'s metric subdifferential, is outside of the scope of this paper.

We adopt the language of “Otto calculus”: consider the manifold $\mf$ of smooth and strictly positive probability densities $\rho:\crc\to\Rp$. 
The tangent space $\tg_\rho\mf$ at any $\rho\in\mf$ is identified with the smooth functions $\varphi:\crc\to\R$ of zero average; 
the infinitesimal motion on $\mf$ induced by $\varphi$ is $\dot\rho=\partial_{xx}(\rho\,\partial_{xx}\varphi)$. Note that by the assumed positivity of $\rho$, any smooth infinitesimal mass-preserving change $\dot\rho:\crc\to\R$ of $\rho$ is uniquely related to a $\varphi\in\tg_\rho\mf$. The manifold $\mf$ is now endowed with a Riemannian structure by introducing a scalar product on each $\tg_\rho\mf$:
\begin{align*}
    (\varphi,\psi)_\rho := \int_\crc \rho\,\partial_{xx}\varphi\,\partial_{xx}\psi\dd x.
\end{align*}
The relation of this Riemannian structure to the metric $\dst$ manifests in the following relation for the length of $\varphi\in\tg_\rho\mf$: 
\begin{align*}
    \|\varphi\|_\rho^2 
    := (\varphi,\varphi)_\rho
    = \int_\crc \rho\,\bigl(\partial_{xx}\varphi\bigr)^2\dd x 
    = \act\bigl(\rho\Leb,\partial_{xx}\varphi\,\rho\Leb\bigr).
\end{align*}
By means of Riesz's representation theorem, the Riemannian structure allows to identify any element~$\xi$ in the cotangent space $\tg^*_\rho\mf$ at $\rho$ with an $v\in\tg_\rho\mf$ such that $\xi=(v,\cdot)_\rho$. In the context at hand, the corresponding linear map $\ons(\rho):\tg^*_\rho\mf\to\tg_\rho\mf$ is called Onsager operator, see e.g.~\cite{LieroMielke}. Here, that operator can be made explicit: the sought $\ons(\rho)\xi\in\tg_\rho\mf$ is the unique function $\psi:\crc\to\R$ of vanishing average such that
\begin{align}
    \label{eq:variational}
    \xi[\varphi] = \int_\crc \psi\,\varphi\dd x \quad \text{for all $\varphi\in\tg_\rho\mf$.}
\end{align}
By definition, the differential $\ddd\fnc(\rho)$ of a given regular functional $\fnc:\mf\to\R$ at some $\rho\in\mf$ lies in the respective cotangent space $\tg^*_\rho\mf$, and the gradient $\grd\fnc(\rho)\in\tg_\rho\mf$ is its Riesz dual with respect to the Riemannian structure. Thus, $\grd\fnc(\rho)=\ons(\rho)\ddd\fnc$, and this is $\fnc$'s variational derivative with respect to $L^2(\crc)$ in accordance to \eqref{eq:variational}; note that the variational derivative can always be chosen to be of zero average.

In summary, the metric gradient $\grd\fnc(\rho)$ of the functional $\fnc$ at $\rho$ is its variational derivative, and the induced infinitesimal motion on $\mf$ by $-\grd\fnc(\rho)$ amounts to 
\begin{align*}
    \partial_t\rho = -\partial_{xx}\big(\rho\,\partial_{xx}\grd\fnc(\rho)\big).
\end{align*}
Solutions to this evolution equation define $\fnc$'s metric gradient flow with respect to $\dst$ on $\mf$. Particularly for $\fnc:=\ent$ from~\eqref{eq:def:ent-fish} with variational derivative $\grd\ent(\rho)=\log\rho$, that gradient flow becomes the DLSS equation~\eqref{eq:DLSS}.

\section{Diffusive transport on the discrete torus}
The goal of this section is to introduce an analogue of the novel distance $\dst$ on the space of probability densities that are piecewise constant with respect to a given mesh on $\crc$.

\subsection{Discretization and notation}\label{ssec:discrete:notation}
We consider an equidistant discretization on the torus with $N \in \mathbb{N}$ intervals of length $\delta = 1/N$. The discrete intervals are labelled by $\kappa \in [N]:=\Z/N\Z\cong\{1,2,\ldots,N\}$ with \change{$\interval_\kappa:=\bigl((\kappa-1/2)\delta,(\kappa+1/2)\delta\bigr)\subset \crc = \R/\Z$ being the $\kappa$th interval on the circle $\crc$ typically identified with $[0,1)$ with cyclic boundary}; here points $x\in\crc$ and indices $\kappa\in[N]$ are always considered as cyclic. Functions $f:[N]\to\R$ are interpreted as piecewise constant, attaining the constant value $f_\kappa:=f(\kappa)$ on $\interval_\kappa$. The integral of $f:[N]\to\R$ is thus
\begin{align*}
	\delta\sum\nolimits_\kappa f_\kappa := \delta\sum_{\kappa\in[N]} f_\kappa 
\end{align*}
and we use the discrete scalar product adapted to the discretization
\begin{equation}\label{eq:delta:scalp}
	\skp{f,g}_\delta := \delta\sum\nolimits_{\kappa} f_\kappa g_\kappa . 
\end{equation}
We write $f_+,f_-:[N]\to\R$ for the left/right translates of $f:[N]\to\R$, i.e., $\bigl(f_\pm\bigr)_\kappa := f_{\kappa\pm1}$ for all $\kappa\in[N]$. The forward/backward difference quotient operators and the discrete Laplacian are defined in the usual way,
\begin{align*}
	\change{\partial^\delta_+ f := \frac{f_+-f}\delta,\qquad\partial^\delta_- f := \frac{f-f_-}\delta } \qquad\text{and}\qquad
	\ddff f := \frac{f_++f_--2f}{\delta^2}.
\end{align*}
Note that $\ddff = \partial^\delta_+\partial^\delta_-= \partial^\delta_-\partial^\delta_+$. For later reference, note further that $\ddff$ is a symmetric linear map on the space of functions $f:[N]\to\R$ with respect to the obvious scalar product, i.e.,
\begin{align}
    \label{eq:ddffsymmetric}
    \skp{\ddff f,g}_\delta = \delta\sum\nolimits_{\kappa}\ddff_\kappa f g_\kappa 
    = \delta\sum\nolimits_{\kappa}f_\kappa \ddff_\kappa g
    = \skp{f,\ddff g}_\delta .    
\end{align}
Moreover, $\ddff$ kernel are the constant function, and $\ddff$ leaves the subspace of functions of zero average invariant. The corresponding restriction is invertible, and by abuse of notation, we shall denote the inverse by $(\ddff)^{-1}$ .

Next, we introduce the set of probability densities on the discrete torus by
\begin{align*}
	\dprb:=\biggl\{ \rho: [N] \rightarrow \mathbb{R}\,\bigg|\,
	\rho_\kappa\ge0\ \text{for all $\kappa\in[N]$},\  \delta\sum\nolimits_{\kappa}\rho_\kappa = 1 \biggr\},
\end{align*}
as well as the subset of positive probability densities,
\begin{align*}
    \dprbplus := \bigl\{ \rho\in\dprb\,\big|\,\rho_\kappa>0\ \text{for all $\kappa\in[N]$}\bigr\}.
\end{align*}
On $\dprb$, we introduce the discretized Hellinger distance for $\rho,\eta\in\dprb$ by
\begin{align}
    \label{eq:def:dhell}
    \dhell(\rho,\eta) := \biggl(\delta\sum\nolimits_\kappa\abs[\big]{\sqrt{\rho_\kappa}-\sqrt{\eta_\kappa}}^2\biggr)^{\!\frac12}.
\end{align}
For later reference, we define the variational derivative of a functional $\mathcal E:\dprb\to\R$ at $\rho^*\in\dprb$ as the dual function $\mathcal E'(\rho^*)$ with respect to the product~\eqref{eq:delta:scalp}; more explicitly,
\begin{align*}
    \big(\mathcal E'(\rho)\big)_\kappa := \frac{\partial \mathcal E(\rho)}{\partial \rho_\kappa}\bigg|_{\rho=\rho^*}.
\end{align*}
In particular, for the discrete entropy functional 
\begin{equation}\label{eqdef:dentropy}
	\dent(\rho) := \delta \sum\nolimits_\kappa \bigl(\rho_\kappa (\log \rho_\kappa - 1)+1\bigr),
\end{equation}
we obtain
\begin{equation}\label{eq:derviv:dentropy}
    \Bigl(\bigl(\dent\bigr)'(\rho^*)\Bigr)_\kappa = \log\rho^*_\kappa.
\end{equation}
Finally, we introduce discrete analogues of the $L^p$- and $H^1$-norms: for $p\in [1,\infty)$, let
\begin{equation}\label{def:LN}
	\dnorm{L^p_N}{v} := \biggl(\delta \sum\nolimits_\kappa v_\kappa^p \biggr)^{\!1/p}\qquad\text{and}\qquad \dnorm{H^1_N}{v} := \biggl( \dnorm{L^2_N}{v}^2 + \delta \sum\nolimits_\kappa \bigl(\partial_+^\delta v)^2 \biggr)^{\!1/2} . 
\end{equation}
We shall further need a discretized \emph{homogeneous $\dot W^{2,\infty}$-norm}, 
\begin{align*}
   \norm{\varphi}_{\dot W^{2,\infty}_N} := \sup_{\kappa\in [N]} \abs*{\ddff_\kappa \varphi} ,
\end{align*}
\change{which thanks to the discrete Poincaré inequality from Proposition~\ref{prop:PI-LSI} below is indeed a norm on all functions $\varphi:[N]\to \R$ with zero average.}
Moreover, we define the dual norm for $f:[N]\to \R$ with average zero by
\begin{align}
    \label{eq:def:dW2infty}
    \norm{f}_{(\dot W^{2,\infty}_N)'} := \sup_{\varphi} \set[\bigg]{ \delta \sum\nolimits_\kappa \varphi_\kappa f_\kappa: \norm{\varphi}_{\dot W^{2,\infty}_N}\leq 1}.
\end{align}

\subsection{Discrete optimal diffusive connection}
This subsection and the next parallel Section \ref{sec:DiffTrans:continuous} in the discrete setting. Several of the technical details simplify thanks to finiteness of the base space.

First, as a discretized analogue of $\connect(\mu^0,\mu^1)$, consider the space $\curves$ of pairs $(\rho^s,\welo^s)_{s\in[0,1]}$ of curves $\rho:[0,1]\to\dprb$ and $\welo:[0,1]\change{\times [N]} \to\R$, where $\rho_\kappa\in H^1([0,1])$ and $\welo_\kappa\in L^2([0,1])$ for each $\kappa\in[N]$, which are subject to the continuity equation
\begin{align}
	\label{eq:dcont999}
	\dot\rho_\kappa^s = \ddff_\kappa\welo^s\quad \text{for all $\kappa\in[N]$ and at almost every $s\in[0,1]$.}
\end{align}
We define a notion of convergence for the set of curves $\curves$.
\begin{definition}[Weak convergence in $\curves$]%
	We say that a sequence $(\rho_n,\welo_n)\in\curves$ converges weakly in $\curves$ towards a limit $(\rho_*,\welo_*)\in\curves$ if $(\rho_n)_\kappa\rightharpoonup(\rho_*)_\kappa$ weakly in $H^1([0,1])$ and $(\welo_n)_\kappa\rightharpoonup(\welo_*)_\kappa$ weakly in $L^2([0,1])$ for each $\kappa\in[N]$.
\end{definition}
Note that $\rho^s_*\in\dprb$, and in particular the conservation of total mass, is part of the definition. Inheritance of the continuity equation \eqref{eq:dcont999} by the limit is automatic. Further note that, thanks to the compact embedding $H^1([0,1])\hookrightarrow C([0,1])$, the curves $\rho_\kappa$ are actually continuous functions, and weak convergence of $(\rho_n,\welo_n)$ to $(\rho_*,\welo_*)$ in $\curves$ implies uniform convergence of $(\rho_n)_\kappa$ to $(\rho_*)_\kappa$, for each $\kappa\in[N]$.

We turn to the discretization of the action functional $\action$. In \eqref{eq:reducedactionfunction}, the pair $(\rho,w)$ is evaluated at the same point; in the discrete setting, where the continuity equation \eqref{eq:dcont999} has a three point stencil, it is natural to consider three point averages of the density $\rho$. 
\begin{definition}[Admissible mobility]\label{def:adm:mob}
A non-negative function $\m: \Rnn^3  \rightarrow [0,+\infty)$ is an \emph{admissible mobility} provided it satisfies the following properties:
\begin{enumerate}
	\item $\m$ is continuous;
    \item $\m$ is concave;
	\item $\m(b,a,c) = \m(b,c,a)$ for $a,b,c \geq 0$;
	\item $\min\{ a,b,c \} \leq \m(b,a,c) \leq \max\{a,b,c\}$ for $a,b,c \geq 0$.
\end{enumerate}
\end{definition}
Moreover, for $\rho\in\dprb$, we abbreviate $\m_\kappa(\rho):=\m( \rho_\kappa, \rho_{\kappa-1}, \rho_{\kappa+1})$. Now, for a pair $(\rho^s,\welo^s)_{s\in[0,1]}\in\curves$, define its action by
\begin{align}
	\label{eq:defac2t}
	\daction\big[(\rho^s,\welo^s)_{s\in[0,1]}\big] := \delta\sum_{\kappa\in\cto}\int_0^1\frc{(\welo_\kappa^s)}{\m_\kappa(\rho^s)} \dd s,
\end{align}
where the square bracket is the quotient's convex and lower semi-continuous relaxation,
\begin{align*}%
	\frc{p}{r} := 
	\begin{cases}
		\frac{p^2}r & \text{if $r>0$}, \\
		0 & \text{if $r=0$ and also $p=0$}, \\
		+\infty & \text{if $r=0$ but $p\neq0$}.
	\end{cases}
\end{align*}
%
\begin{comment}
\begin{remark}[On the choice of $\curves$]
	In the approach of \cite{DNS}, the pairs $(\rho,\welo)$ in the respective definition of the action functional are arbitrary measurable time-dependent Borel and Randon measures, respectively, that are connected by the classical continuity equation. It then needs to be shown, for instance, that finite action implies that $\rho$ can be chosen weakly continuous in time.
	
	Our space $\curves$ --- with $\rho_\kappa\in H^1([0,1])$ and $\welo_\kappa\in L^2([0,1])$ for each $\kappa$ --- is comparatively restrictive. However, that technically convenient restriction comes naturally with our discrete setting: since any $\rho\in\dprb$ trivially satisfies the estimate $\rho_\kappa\le\frac1\delta$, one has
	\begin{align*}
		\delta\sum\nolimits_\kappa\int_0^1 (\welo^s_\kappa)^2\dd s \le \frac 1{\delta}\daction\big[(\rho^s,\welo^s)_{s\in[0,1]}\big]
	\end{align*}
	directly from the definition \eqref{eq:defac2t}. So finiteness of the action immediately implies $\welo_\kappa\in L^2([0,1])$ for each $\kappa$, and a posteriori also $\rho_\kappa\in H^1([0,1])$ by means of \eqref{eq:dcont999}.
\end{remark}
\end{comment}
%
\begin{lemma}[Lower semi-continuity of the action]\label{lem:lsc}
	The action functional associated to an admissible mobility $\m$ is lower semi-continuous with respect to weak convergence in $\curves$.
\end{lemma}
\begin{proof}
	Assume that $(\rho_n,\welo_n)$ converges to $(\rho_*,\welo_*)$ weakly in $\curves$. Then, by definition, we have (more than) weak-$\ast$-convergence in $L^1([0,1])$ of $(\rho_n)_\kappa$ to $(\rho_*)_\kappa$ and of $(\welo_n)_\kappa$ to $(\welo_*)_\kappa$, respectively, for each $\kappa\in[N]$. Since $\m$ is concave, the map
	\begin{equation*}
		(r_-,r_0,r_+,p)\mapsto\frc{p}{\m(r_0,r_-,r_+)}
	\end{equation*}
    is convex, and we conclude, see e.g. \cite{AmbrosioButtazzo88}, that
	\begin{align*}
		\int_0^1\frc{(\welo^s_*)_\kappa}{\m_\kappa(\rho_*^s)}\dd s \le \liminf_n \int_0^1 \frc{(\welo_n^s)_\kappa}{\m_\kappa(\rho_n^s)}\dd s .
	\end{align*}
	And since the \change{limit} inferior of a sum is an upper bound on the sum of the individual limits, this yields the desired estimate
	\begin{equation*}
		\daction\big[(\rho^s_*,\welo^s_*)_{s\in[0,1]}\big]
		\le \delta\sum\nolimits_\kappa \liminf_n \int_0^1\frc{(\welo_n^s)_\kappa}{\m_\kappa(\rho^s_n)}\dd s
		\le \liminf_n \daction\big[(\rho^s_n,\welo^s_n)_{s\in[0,1]}\big]. \qedhere
	\end{equation*}
\end{proof}
For given $\hat\rho,\check\rho\in\dprb$, we denote by $\curves(\hat\rho,\check\rho)$ the subset of $\curves$ consisting of $(\rho^s,\welo^s)_{s\in[0,1]}$ with $\rho^0=\hat\rho$, $\rho^1=\check\rho$. 
\begin{lemma}%
	For an admissible mobility $\m$, every $\rho^0,\rho^1\in\dprb$ can be connected by a locally Lipschitz curve of finite action $(\rho^s)_{s\in[0,1]} \in \curves(\rho^0, \rho^1)$.
 \end{lemma}
\begin{proof}
    Since functions in $\dprb$ are automatically bounded, we can employ a simpler construction as in the proof of the analogous result in the continuous setting in Lemma \ref{lem:H-2}. The idea is to connect both $\rho^0$ and $\rho^1$ to the uniform density on the respective time intervals $[0,1/2]$ and $[1/2,1]$. The difficulty to achieve this with finite action lies in the singular growth of the action functional as the minimum of the density approaches zero. Therefore, we choose the following parametrization:
    \begin{align*}
        \rho^s = \begin{cases}
            (1-4s^2)\rho^0+4s^2 & \text{for $0\le s\le\frac12$}, \\
            (1-4(1-s)^2)\rho^1+4(1-s)^2 & \text{for $\frac12\le s\le 1$}.
        \end{cases}
    \end{align*}
    This is clearly a continuous curve in $\dprb$. For definition of an accompanying function $(\welo_s)_{s\in[0,1]}$, recall that $\ddff$ considered as an endomorphism on the space of functions of zero average is invertible, with inverse denoted by $(\ddff)^{-1}$. Since $\rho^0,\rho^1\in\dprb$ have average one, we may define
    \begin{align*}
        W^0 := (\ddff)^{-1}(\rho^0-1), \quad W^1 := (\ddff)^{-1}(\rho^1-1),
    \end{align*}
    and from here
    \begin{align*}
        \welo^s = \begin{cases}
            8sW^0 & \text{for $0<s<\frac12$}, \\
            8(1-s)W^1 & \text{for $\frac12<s<1$}.
        \end{cases}
    \end{align*}
    We then have that
    \begin{align*}
        \dot\rho^s = 8s(1-\rho^0) = \ddff \welo^s
    \end{align*}
    for $0<s<1/2$, and analogously for $1/2<s<1$. The action of the curve $(\rho_s,\welo_s)_{s\in[0,1]}$ is now easily estimated using that $\rho^s\ge4s^2$ and $\rho^s\ge4(1-s)^2$ for $0<s<1/2$ and for $1/2<s<1$, respectively:
    \begin{align*}
        \int_0^{1}\delta\sum\nolimits_\kappa \frac{(\welo^s_\kappa)^2}{\mob_\kappa(\rho^s)} \dd s
        &\le \int_0^{1/2} \delta\sum\nolimits_\kappa\frac{\big(8sW^0_\kappa\big)^2}{4s^2}\dd s +
        \int_{1/2}^1 \delta\sum\nolimits_\kappa\frac{\big(8(1-s)W^1_\kappa\big)^2}{4(1-s)^2}\dd s \\
        &\le 8\delta\sum\nolimits_\kappa \big(W^0_\kappa\big)^2 + 8 \delta\sum\nolimits_\kappa \big(W^1_\kappa\big),
    \end{align*}
    which is clearly finite.
\end{proof}

\subsection{Discrete diffusive transport distance}
We are now in the position to introduce a discrete analogue of the diffusive transport distance $\dst$. In the following, let an admissible mobility $\m$ be fixed.
\begin{proposition}
    \label{prop:metric}
    The \emph{discrete diffusive transport distance} $\ddst_\mob$, defined between given $\rho^0,\rho^1\in\dprb$ by
	\begin{equation}
        \label{eq:defddst}
		\ddst_\mob(\rho^0, \rho^1) := \inf_{(\rho^s, \welo^s)_{s\in[0,1]} \in \curves(\rho^0,\rho^1)}\daction\big[(\rho^s,\welo^s)_{s\in[0,1]}\big]^{1/2},
	\end{equation}
    is a metric on $\dprb$, and turns $\dprb$ into a geodesic space.    
    Moreover, the distance is bounded from below by the discrete homogeneous negative Sobolev norm of second order, see \eqref{eq:def:dW2infty},
    \begin{equation}\label{eq:comp:W2infty'}
    	\norm{\rho^0-\rho^1}_{(\dot W^{2,\infty}_N)'} \leq \sqrt{3} \ddst_\mob(\rho^0,\rho^1) .
    \end{equation}
\end{proposition}
\begin{proof}
    We first show~\eqref{eq:comp:W2infty'}, which also shows the positivity of the diffusive transport distance. 
    To do so, we let $\eps > 0$ and $(\rho^s, \welo^s)_{s\in[0,1]} \in \curves(\rho^0,\rho^1)$ such that
	\begin{equation*}
		\daction\big[(\rho^s,\welo^s)_{s\in[0,1]}\big]^{\!1/2} < \ddst_\mob(\rho^0, \rho^1) +  \eps \,.
	\end{equation*}
	We obtain for any $\varphi\in W^{2,\infty}_N$ the estimate
	\begin{align*}
		\biggl| \delta \sum\nolimits_\kappa \varphi_\kappa (\rho^0_\kappa - \rho^1_\kappa) \biggr| &= \biggl| \delta \sum\nolimits_\kappa \varphi_\kappa \ddff_\kappa \welo^s \dd s \biggr| \\
		&\leq \biggl( \int_0^1 \delta \sum\nolimits_\kappa \m_\kappa (\rho^s) \left[\ddff_\kappa \varphi\right]^2 \dd s \biggr)^{\!1/2} \biggl( \int_0^1 \delta \sum\nolimits_\kappa \left[\frac{(\welo^s_\kappa)^2}{\m_\kappa(\rho^s)}\right] \dd s \biggr)^{\!1/2} \\ %
        &\leq \norm{\varphi}_{\dot W^{2,\infty}_N} \bra[\bigg]{  \delta \sum\nolimits_\kappa \int_0^1 \mob_\kappa(\rho^s) \dd s}^{\!1/2} \bigl( \ddst_\mob(\rho^0, \rho^1) +  \eps\bigr)\,.
	\end{align*}
    Since, $\m_\kappa(\rho)\le\max(\rho_\kappa,\rho_{\kappa+1},\rho_{\kappa-1})$ by (4) from Definition~\ref{def:adm:mob} of admissible mobility, we get $\delta\sum\nolimits_\kappa\m_\kappa(\rho)\le 3$. Thus, by taking the supremum among all $\varphi\in W^{2,\infty}_N$ with norm bounded by $1$, we get the estimate~\eqref{eq:comp:W2infty'}.
    In particular, if $\ddst_\mob(\rho^0,\rho^1)=0$, then also $\rho^0=\rho^1$.
 
    The triangular inequality relies on a standard argument for parametrisation by arc-length, as e.g. proven in \cite[Lemma 1.1.4]{RoSa03}. Symmetry is obvious. Thus $\ddst_\mob$ is a metric.

    To verify that $(\dprb,\ddst_\mob)$ is a geodesic space, we need to show that, for each choice of $\rho^0,\rho^1\in\dprb$, the infimum in \eqref{eq:defddst} is realized by some curve $(\rho^s,\welo^s)_{s\in[0,1]}\in\curves(\rho^0,\rho^1)$. Since there is at least one curve $(\bar\rho,\bar\welo)\in\curves(\hat\rho,\check\rho)$ of finite action $A<\infty$, and since the action functional is non-negative, $\daction$ has a finite infimum $I\in[0,A]$ on $\curves(\rho^0,\rho^1)$. We need to prove that $I$ is attained. 
	
	Let $(\rho_n,\welo_n)$ a minimizing sequence, with $\daction[(\rho_n,\welo_n)]\le I+1$. Since $\rho_n^s\in\dprb$, we have the trivial upper bound $(\rho^s_n)_\kappa\le\frac1\delta$ and thus $\m_\kappa(\rho^s_n)\le\frac1\delta$. It then follows from
	\begin{align*}
		\int_0^1 (\welo^s_\kappa)^2\dd s 
		\le \frac1\delta\int_0^1\frc{(\welo_n^s)_\kappa}{\m_\kappa(\rho^s_n)}\dd s 
		\le \frac{I+1}{\delta}
	\end{align*}
	that each $(\welo_n)_\kappa$ is $n$-uniformly bounded in $L^2([0,1])$. Using a diagonal argument, there are a (non-relabeled) subsequence and a limit curve $\welo_*$, such that $(\welo_n)_\kappa\rightharpoonup(\welo_*)_\kappa$ in $L^2([0,1])$ for each $\kappa\in[N]$. Now, from the continuity equation $\dot\rho_n = \ddff\welo_n$ and since the boundary values $\rho_n^0=\rho^0$ are fixed, it follows that each $(\rho_n)_\kappa$ converges weakly in $H^1([0,1])$ --- and thus also strongly in $C([0,1])$ --- to a limit $(\rho_*)_\kappa$. By lower semi-continuity, see Lemma~\ref{lem:lsc} above, $(\rho_*,\welo_*)$ is the sought minimizer.
\end{proof}
\begin{remark}
    For the two-point space, $N=2$, we \change{get that $\rho_{\kappa-1}=\rho_{\kappa+1}$ and hence the mobility $\mob(b,a,c)$ becomes a function of two values: $\theta(b,a)=\mob(b,a,a)$. If $\theta$ satisfies suitable assumptions, as e.g.~\cite[Assumption 2.2]{Maas2011},} then $\ddst_\mob$ coincides with the discretized $L^2$-Wasserstein metric that has been introduced independently in \cite{Maas2011,MielkeMC,Chow}, see in particular \cite[Section 2]{Maas2011}.
   	\change{In addition, for the choice~\eqref{eq:def:mob} providing~\eqref{eq:dDLSS0} as the gradient flow of the entropy with respect to $\ddst_\mob$, the mobility $\theta$ is the logarithmic mean, see~\eqref{eq:def:GeoLogMean}, which has been the mobility of choice in~\cite{Maas2011} for the gradient flow formulation of the discretized linear Fokker-Planck equation.}
\end{remark}
Another property, which is preserved by the discretization of the metric, is the interpolation estimate from Lemma~\ref{lem:AubinLions}.
\begin{lemma}\label{lem:AubinLions:discrete}
	If $\rhon\in \AC([0,T],(\dprb,\ddst_\mob))$, then also $\rhon\in \AC([0,T],(\dprb,\dhell))$, with 
	\begin{equation*}%
		\dhell(\rho^\delta(s),\rho^\delta(t)) \leq C \ddst_\mob(\rho^\delta(s),\rho^\delta(t))^{\frac{1}{12}} \qquad\text{for all $0\leq s \leq t \leq T$},
	\end{equation*}
    where $C$ is expressible in terms of $C_\fish:=\sup_{t\in [0,T]} \dfish(\rho_t^\delta)$ alone, without explicit dependence on $\delta$.
\end{lemma}
\begin{proof}
    For the sake of simplicity, we omit the superscript $\delta$ in the proof. In preparation, define the primitive $R(s,t)$ of $\rho(s)-\rho(t)$, that is $\partial^\delta_+ R(s,t) = {\rho(s)}-{\rho(t)}$. Note that $R(s,t)$ is well-defined, since the difference has mean zero by conservation of mass.
    
    The proof goes along the same lines as for Lemma~\ref{lem:AubinLions}. We begin with the Hölder estimate~\eqref{eq:HellHoelder} and apply the elementary inequality \eqref{eq:elementary}. In terms of the primitive $R(s,t)$, we arrive after a summation by parts and Cauchy-Schwarz at
	\begin{align}
		\dhell(\rho(s),\rho(t)) &\leq \pra[\bigg]{ \delta \sum\nolimits_\kappa  \bigl(\sqrt{\rho(s)}-\sqrt{\rho(t)}\bigr)  \partial^\delta_+R(s,t)}^{\!2/3} \notag \\
		&\leq \pra[\bigg]{ \delta \sum\nolimits_\kappa  \abs*{\partial^\delta_- \bra*{ \sqrt{\rho(s)}-\sqrt{\rho(t)}}}^2}^{\!1/3} \pra[\bigg]{ \delta \sum\nolimits_\kappa  R(s,t)^2}^{\!1/3}  .  \label{eq:dhell:interpol}
	\end{align}
	Note, that 
	$\abs[\big]{\partial^\delta_- \bra*{ \sqrt{\rho(s)}-\sqrt{\rho(t)}}}^2 \leq 2 \abs[\big]{\partial^\delta_- \sqrt{\rho(s)}}^2 +2 \abs[\big]{\partial^\delta_- \sqrt{\rho(t)}}^2$ and so the first term on the right-hand side of~\eqref{eq:dhell:interpol} is bounded in terms of $C_\fish$. %
	For the second term, we establish the discrete analog of inequality~\eqref{eq:est:H-1:W2'} from Lemma~\ref{lem:AubinLions} taking the form
	\begin{equation}\label{eq:est:H-1:W2':discrete}
		\pra[\bigg]{ \delta \sum\nolimits_\kappa  R(s,t)^2}^{\!1/2} \leq C \norm{\rho(s)-\rho(t)}_{(W^{2,\infty}_N)'}^{\frac{1}{4}} ,
	\end{equation}
	which with the bound estimate~\eqref{eq:comp:W2infty'} provides the conclusion. 
	For proving~\eqref{eq:est:H-1:W2':discrete}, we rewrite the $L^2_N$-norm of $R(s,t)$ as an $H^{-1}$-norm in dual formulation,
	\begin{equation*}
		\pra[\bigg]{ \delta \sum\nolimits_\kappa  R(s,t)^2}^{\!1/2} = \sup_{\varphi}\set[\bigg]{ \delta \sum\nolimits_\kappa \varphi_\kappa \bra*{ \rho_\kappa(s) - \rho_\kappa(t)} : \delta \sum\nolimits_\kappa \abs*{\partial^\delta_+ \varphi}^2 \leq 1} .
	\end{equation*}
	Fix $\varphi:[N]\to \R$ and consider its convolution-type regularization $\varphi^\eps_{\lambda} = \delta \sum\nolimits_\kappa K^{\eps,\delta}_{\lambda-\kappa} \varphi_\kappa$ with the discretized van Mises kernel $K^{\eps,\delta}_\kappa = c_{\eps,N}^{-1} \exp\bra*{ \eps^{-1} \cos(2\pi \delta \kappa)}$, normalized by $c_{\eps,N} = \sqrt{2\pi \eps} + O_\eps(\eps) + o_\delta(1)$. Then, can estimate as in~\eqref{eq:H1:W2infty} \change{by a triangle and duality inequality to obtain}
	\begin{equation*}
		\delta \sum\nolimits_\kappa \varphi_\kappa \bra*{ \rho_\kappa(s) - \rho_\kappa(t)} \leq 2 \sup_\kappa \abs*{ \varphi_\kappa- \varphi_\kappa^\eps} + \norm{\varphi^\eps}_{W^{2,\infty}_N} \norm{\rho(s)-\rho(t)}_{(W^{2,\infty}_N)'} \,,
	\end{equation*}
	\change{where we have used that $\rho(s),\rho(t)\in \Prob^\delta$ have unit mass.}
	By definition of $K^{\eps,\delta}$, the rest of the scaling arguments in the proof of Lemma~\ref{lem:AubinLions} for establishing~\eqref{eq:est:H-1:W2':discrete} carry over. In particular, the constant $C$ in~\eqref{eq:est:H-1:W2':discrete} is uniform in $\delta$. 
\end{proof}

\subsection{Gradient flows with respect to the discrete diffusive transport distance}%
We close with a formal consideration of gradient flows on the metric space $(\dprb,\ddst_\mob)$. For that, we consider the $N$-dimensional smooth manifold $\dprbplus\subset\dprb$ of \emph{strictly positive} densities. A tangent vector at $\rho$ is associated to a function $\varphi:[N]\to\R$ of zero average, the corresponding infinitesimal tangential motion on $\dprbplus$ being given by $\dot\rho=\ddff(\mob(\rho)\ddff\varphi)$. By the properties of the discrete Laplacian, any mass-preserving motion $\dot\rho$ possesses a unique such representation.

The metric $\ddst_\mob$ makes $\dprbplus$ a Riemannian manifold in the classsical sense. The metric tensor at a given $\rho$ corresponds to the scalar product \begin{align*}
    (\varphi,\psi)_\rho := \delta\sum\nolimits_\kappa \m_\kappa(\rho)\,\ddff_\kappa\varphi\,\ddff_\kappa\psi.      
\end{align*}
The differential $F'(\rho)\in\tg^*_\rho\dprbplus$ of a smooth function $F:\dprbplus\to\R$ can be identified with the function $\xi:[N]\to\R$ consisting of partial derivatives,
\[ \xi_\kappa = \frac{\partial F}{\partial\rho_\kappa}(\rho)-\sigma, \]
with $\sigma\in\R$ chosen such that $\xi$ has zero average. Accordingly, the gradient $\nabla F(\rho)$ of $F$ is defined as the unique $\psi\in\tg_\rho\dprbplus$ with $(\psi,\cdot)_\rho=\dn_\rho F$. More explicitly, $\psi$ satisfies
\begin{align*}
    \delta\sum\nolimits_\kappa \m_\kappa(\rho) \,\ddff_\kappa\psi\,\ddff_\kappa\varphi = \delta\sum\nolimits_\kappa \frac{\partial F(\rho)}{\partial\rho_\kappa} \ddff_\kappa(\mob(\rho)\ddff\varphi) 
\end{align*}
for all $\varphi:[N]\to\R$ of zero average. This implies that $\psi=F'(\rho)$. Consequently, $F$'s gradient flow on the Riemannian manifold $(\dprbplus,\ddst_\mob)$ is given by
\begin{equation}\label{eq:general_diffusive_GF}
	\dot \rho = - \ddff\big( \m(\rho)\,\ddff F'(\rho)\big) . 
\end{equation}
Note that $\ddst_\mob$ is a metric on all of $\dprb$, but the metric tensor degenerates for $\rho\in\dprb\setminus\dprbplus$, i.e., when $\rho_\kappa$ vanishes at some $\kappa\in[N]$. Thus, $F$'s gradient is not defined at those $\rho$, even if $F$ is smooth everywhere on $\dprb$. 

\begin{remark}[Interpolation in the spirit of Scharfetter--Gummel~\cite{SG1969}]%
	A possible mobility function for a discrete gradient flow formulation of~\eqref{eq:DLSS} can be obtained by following the construction after Scharfetter--Gummel~\cite{SG1969} (see also~\cite{il1969difference}), who constructed a two-point flux interpolation for diffusion equations.
	This idea became the basis for numerous other generalizations, e.g.\ for equations with nonlinear diffusion \cite{bessemoulin2012finite, EymardFuhrmannGaertner2006, jungel1995numerical} or aggregation-diffusion equations~\cite{SchlichtingSeis2022,HraivoronskaSchlichtingTse2023}. 
	For the static quantum drift diffuion~\cite{ChainaisHillairetGisclonJuengel2010} obtains a finite volume discretization based on a rewriting as a system of two second order equations.
	
	In a similar spirit, we can obtain a finite volume discretization by considering the semi-discrete second-order continuity equation
	$\dot \rho = -\ddff j$, where the flux $j_\kappa = j[\rho_{\kappa-1},\rho_\kappa,\rho_{\kappa+1}]$ shall be a \emph{three-point} approximation of the DLSS-flux $\rho \partial_{xx} \log \rho$.
	Following the idea of Scharfetter--Gummel~\cite{SG1969}, we ask the flux $j$ to solve the second order cell problem
	\[
	j = \rho \partial_{xx}\log \rho \ \text{ on } (-\delta,\delta) ,\qquad\text{with}\qquad  \rho(-\delta)=\rho_{\kappa-1}, \quad \rho(0) =\rho_\kappa ,\quad \rho(\delta) = \rho_{\kappa+1} . 
	\]
	The identity $\partial_{xx} \log \cos x = - \cos(x)^{-2}$ and the affine invariance of the solutions imply that 
	for arbitrary $x_0,\omega\in \R$ the general solution is given by $\rho(x) =\frac{j}{2\omega}\cos\bra*{\sqrt{\omega}( x +x_0)}^2$ (understood complex for $\omega<0$). 
	The parameters $x_0, \omega ,j$ must be solved in terms of the boundary values, respectively.
	Although, the procedure leads to an implicit equation, we believe it can provide a very accurate numerical finite volume scheme with good structure preserving properties being also applicable to the DLSS equation with external potentials or in self-similar variables. For now, we study the explicit mobility function~\eqref{eq:def:mob} below, which has already surprising structure preserving properties and leave more general ones for future investigation.
\end{remark}

\section{Variational discretization of the DLSS equation}
\label{sct:QDDnum}
The motivation for our discretization \eqref{eq:dDLSS0} of the DLSS equation is that it is the gradient flow on the finite-dimensional Riemannian manifold $(\dprb,\ddst_\mob)$ of a discretized entropy functional $\dent$ with respect to the discrete diffusive transport metric $\ddst_\mob$ for an appropriate mobility $\mob$. 

\subsection{Mobility function}
For the discretization of \eqref{eq:DLSS}, we use the mobility $\mob:\Rnn^3\to\Rnn$ given by
\begin{align}\label{eq:def:mob}
	\mob(r_0,r_+,r_-) :=
	\begin{cases}
		\frac{\sqrt{r_{\smash+}r_{\smash-}}-r_0}{\log\sqrt{r_{\smash+}r_{\smash-}}-\log r_0} & \text{if $r_0,r_+,r_->0$ and $r_+r_-\neq r_0^2$}, \\
		r_0 & \text{if $r_+r_-=r_0^2$}, \\
		0 & \text{if $r_0=0$ or $r_+=0$ or $r_-=0$}.
	\end{cases}
\end{align}
Notice that $\mob$'s definition is independent of the step size $\delta$. Further, $\mob$ possesses a representation as a nested mean of its arguments: 
recall the definitions of the geometric and the logarithmic mean, 
\begin{equation}\label{eq:def:GeoLogMean}
\mathbf{G}(a,b) := \sqrt{ab}, \qquad
\logmean(a,b) := \begin{cases} 
	\frac{a-b}{\log a-\log b}, & a\ne b ;\\
	a , & a=b ,
\end{cases}	 \qquad\text{for $a,b>0$.}
\end{equation}
Then, we have the identity $\mob(r_0,r_+,r_-) = \logmean\bigl(r_0,\mathbf{G}(r_+,r_-)\bigr)$.
\begin{remark}
    The above construction of an admissible mobility out of two mean functions is not restricted to the concrete choice of the two-point means $\mathbf{G}$ and $\logmean$ in~\eqref{eq:def:GeoLogMean} but generalizes to other means under suitable concavity and monotonicity assumptions.
\end{remark}
\begin{lemma}\label{le:threepointmean}
	$\mob$ is admissible in the sense of Definition~\ref{def:adm:mob}. In addition, $\mob$ is one-homogeneous, and is zero if and only if one of its arguments are zero. Moreover,
	\begin{align}
		\label{eq:logmeancompare}
		\logmean\big(r_0,\min\{r_+,r_-\}\big) \le\mob(r_0,r_+,r_-) \le \logmean\big(r_0,\max\{r_+,r_-\}\big).
	\end{align}   
\end{lemma}
\begin{proof}
    Continuity and one-homogeneity are obvious from the definition \eqref{eq:def:mob}. It remains to verify concavity and inequality \eqref{eq:logmeancompare} --- which implies in particular the inequality from Definition~\ref{def:adm:mob}, and also that $\mob(r_0,r_+,r_-)=0$ if and only if $r_0r_+r_-=0$.
    
	For the proof of concavity, use that by one-homogeneity
	\[ r_0\mob(1,\sigma_+,\sigma_-) = \mob(r_0,r_0\sigma_+,r_0\sigma_-) \qquad\text{for } r_0,\sigma_+,\sigma_-\in \Rnn. \]
	Defining
	\[ \mobb(\sigma_+,\sigma_-):=\mob(1,\sigma_+,\sigma_-) = \frac{\sqrt{\sigma_{\smash+}\sigma_-}-1}{\log(\sigma_+\sigma_-)}, \]
	we find with $r=(r_0,r_+,r_-)\in \Rnn^3$ and $\sigma=(\sigma_+,\sigma_-)\in \Rnn^2$ that
	\begin{align*}
		\nabla_r^2\mob(r)
		= \nabla_r^2\left[r_0\mobb\left(\frac{r_+}{r_0},\frac{r_-}{r_0}\right)\right]
		= A(r)^T\nabla^2_\sigma\mobb\left(\frac{r_+}{r_0},\frac{r_-}{r_0}\right)A(r)
		\ \text{ with }\ A(r)=\begin{pmatrix}
			-\frac{r_+}{r_0} & 1 & 0 \\
			-\frac{r_-}{r_0} & 0 & 1
		\end{pmatrix}.
	\end{align*}
	It follows that $\mob$ is concave if $\mobb$ is. To verify concavity of $\mobb$, write
	\[ \mobb(\sigma) = f(\sigma_+\sigma_-) \quad\text{with}\quad f(z) = \frac{\sqrt{z}-1}{\log z} \]
	and notice that
	\[ \nabla^2_{(u,v)}\mobb(u,v) = f''(uv)\,\begin{pmatrix}v\\u\end{pmatrix}\begin{pmatrix}v\\u\end{pmatrix}^T,\]
	hence $\mobb$ is concave if $f$ is. Concavity of $f$ follows from the representation
	\begin{align*}
		f(z) = \int_0^1 z^{s/2}\dd s.
	\end{align*}
	That is, $f$ is an average of the concave functions $z\mapsto z^{s/2}$ for $0\le s\le1$, and therefore concave as well. 
	The comparison \eqref{eq:logmeancompare} follows by monotonicity of the logarithmic mean, and since $\min\{r_+,r_-\}\le\sqrt{r_{\smash+}r_-}\le\max\{r_+,r_-\}$.
\end{proof}
According to Proposition~\ref{prop:metric}, the mobility $\mob$ defined in~\eqref{eq:def:mob} gives rise to a discrete diffusive transport metric $\ddst_\mob$ on $\dprb$. The metric gradient flow~\eqref{eq:general_diffusive_GF} with respect to $\ddst_\mob$ for the discrete entropy functional $\dent$, see \eqref{eqdef:dentropy}, takes the form
\begin{align}\label{eq:dDLSS:GF}
	\dot\rho_\kappa
	= -\ddff_\kappa\bigl(\mob(\rho)\ddff\log\rho\bigr), \qquad\text{ for } \kappa\in [N].
\end{align}
By our choice of $\mob$, this equation agrees with the proposed scheme~\eqref{eq:dDLSS0}. Using the notations introduced in Section \ref{ssec:discrete:notation}, the latter is written in the compact form
\begin{align}
	\label{eq:dDLSS}
	\dot\rho_\kappa = - 2 \ddff_\kappa\biggl(\frac{\sqrt{\rho_+\rho_-}-\rho}{\delta^2}\biggr) \qquad\text{for $\kappa \in [N]$}.
\end{align}
\begin{remark}
    Our particular choice of mobility $\mob$ is motivated by the following alternative representation of \eqref{eq:dDLSS:GF}:
    \begin{align}
        \label{eq:dDLSS:savare}
        \frac{\dd}{\dn t}\change{\sqrt{\rho_\kappa }}
        = -\ddff_\kappa\ddff\sqrt\rho + \frac{\big(\ddff_\kappa\sqrt\rho\big)^2}{\sqrt{\rho_\kappa}}.
    \end{align}
    Equivalence of \eqref{eq:dDLSS:GF}, \eqref{eq:dDLSS}, and \eqref{eq:dDLSS:savare} for positive solutions is easily checked by direct computation. The form \eqref{eq:dDLSS:savare} is the direct analogue of \eqref{eq:dlssVH} and facilitates the proof of contractivity in the discrete Hellinger distance, see Lemma \ref{lem:dDLSShell} below.
\end{remark}

\subsection{Properties of the discretization}\label{sct:dQDDprop}
Below, we collect the properties stated in Result~\ref{Result:Scheme} that are satisfied by solutions to \eqref{eq:dDLSS}, which are then proven in the rest of Section \ref{sct:QDDnum}. First, we establish global well-posedness of \eqref{eq:dDLSS} for arbitrary non-negative initial data.
\begin{proposition}
    \label{prp:dQDDx1}
    The flow defined by \eqref{eq:dDLSS} on $\dprbplus$ is global and possesses a unique continuous extension to a flow on $\dprb$.
\end{proposition}
The proof of Proposition \ref{prp:dQDDx1} rests on two auxiliary results: global existence follows from a dissipation property, and the continuous extension to $\dprb$ is obtained from contractivity in a discretized Hellinger distance \eqref{eq:def:dhell}.  The corresponding auxiliary results are:
\begin{lemma}
    \label{lem:dDLSSpos}
    The heat capacity defined by 
    \begin{equation*}%
	   \dlyp(\rho):=-\delta\sum\nolimits_\kappa \log\rho_\kappa.        
    \end{equation*}
    is non-increasing along positive solutions to \eqref{eq:dDLSS} and satisfies
    \begin{equation}\label{eq:dlyp:diss}
    	\dlyp(\rho(t)) + \frac12 \int_0^t  \delta \sum\nolimits_\kappa \abs*{\ddff_\kappa \log \rho(s)}^2 \dd s \leq \dlyp(\rho(0)) \qquad\text{for any } t >0  .
    \end{equation}
    Consequently, for each $\rho(0)\in\dprbplus$, there is a unique positive solution $\rho:[0,\infty)\to\dprbplus$ to \eqref{eq:dDLSS}, and that satisfies the a priori bound
   	\begin{equation}\label{eq:sup:lyp:bound}
   		\int_0^t \sup_{\kappa,\lambda} \abs*{ \log \rho^\delta_\kappa(s) - \log \rho^\delta_\lambda(s)}^2 \dd s \leq \change{\frac{C_{\textup{PI}}}{2}} \dlyp(\rho^\delta(0)) \,,
   	\end{equation}
   	\change{with $C_{\textup{PI}}$ the Poincaré constant from Proposition~\ref{prop:PI-LSI}.}
\end{lemma}
\begin{lemma}
    \label{lem:dDLSShell}
    For any two positive solutions $\rho$ and $\eta$ to \eqref{eq:dDLSS}, their distance $\dhell(\rho(t),\eta(t))$ is non-increasing in time. 
\end{lemma}
By the gradient flow structure, solutions to \eqref{eq:dDLSS} dissipate the discretized entropy $\dent$. The rate of dissipation can be quantified, and this will be essential for the convergence proof in Section \ref{sct:dQDDconv}.
\begin{lemma}
    \label{lem:dDLSSent}
    Any positive solution to \eqref{eq:dDLSS} satisfies
    \begin{align}
        \label{eq:dissipation:bound}
	        -\frac{\dn}{\dd t}\dent(\rho) 
            =4\delta \sum\nolimits_\kappa \left(\frac{\sqrt{\rho_{\kappa+1}\rho_{\kappa-1}}-\rho_\kappa}{\delta^2}\right)\left(\frac{\log\sqrt{\rho_{\kappa+1}\rho_{\kappa-1}}-  \log \rho_\kappa}{\delta^2}\right)
	        \ge \delta\sum\nolimits_\kappa \big(\ddff\!\sqrt\rho\big)^2.
    \end{align}
\end{lemma}
The continuum analog of the inequality has been crucial for proving existence to~\eqref{eq:DLSS},  see~\cite[Eq.(1.82)]{GianazzaSavareToscani2009} and \cite[Eq.(1.3)]{JuengelMatthes2008}.
Finally, we establish also a relation to the gradient flow formulation of~\eqref{eq:DLSS} in the $L^2$-Wasserstein metric given in \cite{GianazzaSavareToscani2009}. This implies en passant that the discretized Fisher information,
\begin{align}\label{eq:def:dfish}
   \dfish(\rho) := 2 \delta \sum\nolimits_\kappa \left(\frac{\sqrt{\rho_{\kappa+1}}-\sqrt{\rho_\kappa}}\delta\right)^2,
\end{align}
is another Lyapunov function for \eqref{eq:dDLSS}. We do not repeat the construction of $L^2$-Wasserstein distance on $\dprb$ from \cite{Maas2011}, but just recall that any symmetric concave function $\geom:\Rnn\times\Rnn\to\Rnn$ gives rise to such a discretized metric.
\begin{lemma}
    \label{lem:dDLSSfish}
    Let $\dwass$ be the discretized $L^2$-Wasserstein distance for $\geom(a,b)=\sqrt{ab}$ in the sense of~\cite{Maas2011}. Then the positive solutions to \eqref{eq:dDLSS} form a gradient flow of $\dfish$ with respect to $\dwass$, that is
    \begin{align*}%
    	\dot\rho = 
    	-\partial_-^\delta \bra*{\geom(\rho)\,\partial_+^\delta \bra*{\frac{\partial\dfish}{\partial\rho}}} = 2 \partial_-^\delta\bra*{ \sqrt{\rho \rho_{\smash+}} \; \partial_+^\delta\bra*{ \frac{\Delta^\delta \sqrt{\rho}}{\sqrt{\rho}}}}.
    \end{align*}
\end{lemma}
A posteriori, we derive some universal bounds on the decay of $\dent$ and $\dfish$.
\begin{lemma}\label{lem:uniform:time:bounds}
	Let $C_{\textup{LSI}}$ the logarithmic \change{Sobolev} constant in~\eqref{eq:LSI:d}, which can be chosen uniform in $\delta$, then any solution $\rhon$ \change{constructed in Proposition~\ref{prp:dQDDx1}} satisfies, for each $t>0$,
	\begin{align}
		\label{eq:unident}
		\dent(\rhon(t)) &\le 4 C_{\textup{LSI}}^2 t^{-1} , \\
		\label{eq:unidfish}
		\dfish(\rhon(t)) &\le 8 C_{\textup{LSI}} t^{-1}, \\
		\label{eq:unidest}
		\int_t^\infty \delta\sum\nolimits_\kappa \big(\ddff_\kappa\sqrt{\rho(s)}\big)^2\dd s &\le 4 C_{\textup{LSI}}^2 t^{-1}.
	\end{align}
	Likewise, for every $r > 0$, there exists an explicit constant $C_r>0$, such that any solution $\rhon$ \change{constructed in Proposition~\ref{prp:dQDDx1}} satisfies
		\begin{equation}\label{eq:scheme:shorttime}
			\dent(\rhon(t)) \leq \change{(C_r t)^{-1/r} } \qquad\text{ for all } t > 0 . 
		\end{equation}
\end{lemma}
We conclude this section by pointing out yet another variational structure of our discretization.
\begin{remark}[Generalized gradient structures]\label{rem:GGF}
	The heuristic derivation from~\eqref{eq:heuristic:derivation} suggests that the scheme~\eqref{eq:dDLSS} has a generalized gradient structure~\cite{MielkePeletierRenger2014,PRST22,HoeksemaTse2023,Hoeksema2023}. More specifically, we show that the scheme~\eqref{eq:dDLSS:GF} has a gradient structure in continuity equation format~\cite[Definition~1.1]{PeletierSchlichting2022} with building blocks $([N],[N],\ddff,\dent,\mathcal{R}^\delta)$. Hereby, the role of the continuity equation on the discretized torus $[N]$ is the discrete second order equation~\eqref{eq:dcont999}: $\dot\rho_\kappa^s = \ddff_\kappa\welo^s$ and we identitfy vertices and edges in the present setting. 
	
	The diffusive flux $\welo$ to arrive at~\eqref{eq:dDLSS} can be rewritten as
	\[
	  \welo = -\frac{2}{\delta^2} \bra*{\sqrt{\rho_+ \rho_-} -\rho}
	  =
	  -\frac{2}{\delta^2}
	  \rho \pra*{ \exp\bra*{ \frac{\delta^2}{2} \ddff \log \rho}-1} .
	\]
	The last identity can be encoded through the kinetic relation driven by the force $\ddff (\dent)'$ as
	\[
	   \welo = \mathrm{D}_2 \mathcal{R}^{\delta,*}\bra*{\rho,-\ddff (\mathcal{H}^\delta)'(\rho)}
	\]
	where the dual dissipation potential is given by 
	\begin{equation*}%
		\mathcal{R}^{\delta,*}(\rho,\xi) :=   \delta\sum\nolimits_\kappa  \rho_\kappa \frac{2}{\delta^2}\pra*{\frac{2}{\delta^2} \bra*{\exp\bra*{-\frac{\delta^2}{2} \xi_\kappa}-1}+\xi_\kappa} .
	\end{equation*}
	In our setting, we use as pairing between forces and fluxes $\skp{\xi,\welo}_{\delta} = \delta\sum\nolimits_\kappa \xi_\kappa \welo_\kappa$, which also formally passes to the limit $\delta\to 0$. Therewith, the primal dissipation functional as pri-dual of $\mathcal{R}^{\delta,*}$ defined by $\mathcal{R}^{\delta}(\rho,\welo) = \sup_{\xi} \set*{ \skp{\xi,\welo}_{\delta} - \mathcal{R}^{\delta,*}(\rho,\xi)}$ takes the explicit form
 \begin{equation}\label{eq:def:R}
		\mathcal{R}^{\delta}(\rho,\welo) = \delta \sum\nolimits_\kappa\rho_\kappa  
  \eta\bra*{1-\frac{\delta^2}{2}\frac{\welo_\kappa}{\rho_\kappa}} \quad\text{with }  \eta(s) = s \log s - s  +1,
	\end{equation}	
	provided that $\rho_\kappa \geq \frac{\delta^2}{2} \welo_\kappa$ for $\kappa\in [N]$ and $+\infty$ else. 
    We note that we have the expansion
    \[
    \rho_\kappa  
  \eta\bra*{1-\frac{\delta^2}{2}\frac{\welo_\kappa}{\rho_\kappa}}  = \frac{\welo_\kappa^2}{2\rho_\kappa}+ O(\delta) .
    \]
	Hence, we formally get that~\eqref{eq:def:R} provides a discrete approximation for the continuous action density~\eqref{eq:def:action_density}. Indeed, provided that $\rho^\delta \to \rho$ and $\welo^\delta \to \welo$ in some suitable way, we expect that
	\[
	\mathcal{R}^{\delta}(\rho^\delta,\welo^\delta) \to \frac{1}{2} \act(\rho,\welo)\qquad\text{ as } \delta\to 0. 
	\]
\end{remark}

\subsection{Proof of Lemma \ref{lem:dDLSSpos}}
The differential equation \eqref{eq:dDLSS} has a locally Lipschitz continuous right-hand side, so by the Picard-Lindel\"of theorem, there exist a time horizon $T\in\Rp\cup\{+\infty\}$ and a continuously differentiable curve $\rho:[0,T)\to\Rp^N$ such that $\rho$ is the unique solution to the corresponding initial value problem. And moreover, unless $T=+\infty$, there exists a sequence $t_k\uparrow T$ such that $\rho(t_k)$ escapes $\dprbplus$, that is $\rho_\kappa(t_k)\to0$ or $\rho_\kappa(t_k)\to\infty$ as $k\to\infty$ for some index $\kappa$. Finally, by differentiability in time and periodicity in $\kappa$, it follows that
\begin{align*}
	\frac{\dn}{\dd t} \delta\sum\nolimits_\kappa\rho_\kappa 
		= - 2\delta\sum\nolimits_\kappa \ddff\biggl(\frac{\sqrt{\rho_{\smash+}\rho_-}-\rho_0}{\delta^2}\biggr)
		= 0,
\end{align*}
hence $\delta\sum\nolimits_\kappa\rho_\kappa(t)=1$ for all $t\in[0,T)$, and consequently $\rho(t)\in\dprbplus$. 
	
To prove that actually $T=+\infty$, i.e. that $\rho$ is a global solution, it suffices to derive a finite upper and a positive lower bound on each component $\rho_\kappa(t)$. The bound above follows immediately from $\rho(t)\in\dprb$:
\begin{align*}
	\delta\rho_\kappa(t) 
	\le \delta\sum\nolimits_{\kappa'}\rho_{\kappa'}(t) = 1 
	\quad\Rightarrow\quad \rho_\kappa(t)\le N.
\end{align*}
The lower bound will follow from monotonicity of $\dlyp$. Indeed, provided that $\dlyp(\rho(t)\change{)}\le\dlyp(\rho(0))$, we then may conclude that
\[ -\delta\log\rho_\kappa(t) 
	= \delta\sum_{\kappa'\neq\kappa}\log \rho_{\kappa'}(t) + \dlyp(\rho(t))
	\le \frac1e + \dlyp(\rho(0)), \]
where we have used that, by Jensen's inequality, for any $\rho\in\dprb$
\begin{align*}
		\delta\sum_{\kappa'\neq\kappa}\log\rho_{\kappa'}
		&= \frac{N-1}N \biggl(\frac1{N-1}\sum_{\kappa'\neq\kappa}\log\rho_{\kappa'}\biggr)
		\le \frac{N-1}N \log\biggl(\frac1{N-1}\sum_{\kappa'\neq\kappa}\rho_{\kappa'}\biggr)\\
		&\le \frac{N-1}N \log\biggl(\frac{N}{N-1}\delta\sum\nolimits_\kappa \rho_\kappa\biggr)
		\le \frac{N-1}{N}\log\frac{N}{N-1} \le \frac1e.
\end{align*}
To prove the desired monotonicity of $\dlyp(\rho(t))$, we compute its time derivative:
\begin{align}
		\nonumber
		-\frac{\dn}{\dd t}\dlyp(\rho)
		&=\delta\sum\nolimits_\kappa \frac{\dot\rho}{\rho} 
		= -2\delta \sum\nolimits_\kappa \frac1\rho\,\ddff\biggl(\frac{\sqrt{\rho_{\smash+}\rho_{\smash-}}-\rho}{\delta^2}\biggr)
		= -2\delta\sum\nolimits_\kappa \frac{\sqrt{\rho_{\smash+}\rho_{\smash-}}-\rho}{\delta^2} \ddff\biggl(\frac1\rho\biggr) \\
		\nonumber
		&= -2\frac{\delta}{\delta^4}\sum\nolimits_\kappa \big(\sqrt{\rho_{\smash+}\rho_{\smash-}}-\rho\big) \left(\frac1{\rho_+}+\frac1{\rho_-}-\frac2{\rho}\right) \\
		\label{eq:logsum}
		&= 4\delta^{-3}\sum\nolimits_\kappa \left(
		\sqrt{\frac{\rho_{\smash+}\rho_-}{\rho^2}}-1\right)
		+2\delta^{-3}\sum\nolimits_\kappa\left(
		\frac{\rho}{\rho_+}+\frac{\rho}{\rho_-}
		-\sqrt{\frac{\rho_+}{\rho_-}}
		-\sqrt{\frac{\rho_-}{\rho_+}}
		\right).
\end{align}
Above, we have used the symmetry of the discrete Laplacian, see \eqref{eq:ddffsymmetric}. For estimation of the first sum in \eqref{eq:logsum} above, we apply Jensen's inequality for sums with the convex exponential function,
\begin{align*}
		\delta\sum\nolimits_\kappa
		\sqrt{\frac{\rho_+\rho_-}{\rho^2}}
		&=\frac2N\sum\nolimits_\kappa \exp\bra*{\frac12\bigl(\log\sqrt{\rho_+}+\log\sqrt{\rho_-}-2\log\sqrt{\rho}\bigr)} \\
		&\ge \exp\biggl(\frac1N\sum\nolimits_\kappa\bigl(\log\sqrt{\rho_+}+\log\sqrt{\rho_-}-2\log\sqrt{\rho}\bigr)\biggr)
		= \exp(0) = 1,
\end{align*}    
where we have used that the sum in the exponent vanishes, again thanks to periodicity. Consequently, the first sum in \eqref{eq:logsum} is non-negative. Exploiting periodicity once again, we can symmetrize the expression in the second sum in \eqref{eq:logsum} as follows:
\begin{align*}
		\delta\sum\nolimits_\kappa \left(\frac{\rho}{\rho_+}+\frac{\rho}{\rho_-}\right)
		= \frac\delta2\sum\nolimits_\kappa \left(
		\frac{\rho}{\rho_+}+\frac{\rho_-}{\rho}+
		\frac{\rho_+}{\rho}+\frac{\rho}{\rho_-}\right)
\end{align*}
Consequently,
\begin{align*}
	 \MoveEqLeft\delta\sum\nolimits_\kappa\left(
		\frac{\rho}{\rho_+}+\frac{\rho}{\rho_-}
		-\sqrt{\frac{\rho_+}{\rho_-}}
		-\sqrt{\frac{\rho_-}{\rho_+}}
		\right)\\
		&=\frac\delta2\sum\nolimits_\kappa\left(
		\frac{\rho}{\rho_+}+\frac{\rho_-}{\rho}+
		\frac{\rho_+}{\rho}+\frac{\rho}{\rho_-}
		-2\sqrt{\frac{\rho_+}{\rho}}\sqrt{\frac{\rho}{\rho_-}}
		-2\sqrt{\frac{\rho_-}{\rho}}\sqrt{\frac{\rho}{\rho_+}}
		\right) \\
		&=\frac\delta2\sum\nolimits_\kappa\bra[\bigg]{
		\left[\sqrt{\frac{\rho_+}{\rho}}-\sqrt{\frac{\rho}{\rho_-}}\right]^2
		+\left[\sqrt{\frac{\rho_-}{\rho}}-\sqrt{\frac{\rho}{\rho_+}}\right]^2}
		\ge0.
\end{align*}
Thus the sum in \eqref{eq:logsum} is non-negative, and so $t\mapsto\dlyp(\rho(t))$ is monotonically decreasing and satisfies
\begin{equation*}
	\dlyp(\rho(t)) + \int_0^t \delta\sum\nolimits_\kappa \frac{1}{\delta^4}\bra[\bigg]{
	\left[\sqrt{\frac{\rho_+}{\rho}}-\sqrt{\frac{\rho}{\rho_-}}\right]^2
	+\left[\sqrt{\frac{\rho_-}{\rho}}-\sqrt{\frac{\rho}{\rho_+}}\right]^2} \dd s \leq \dlyp(\rho(0)), \qquad\text{for any } t>0. 
\end{equation*}
The proof of~\eqref{eq:dlyp:diss} is finished by setting $u:=\sqrt{\rho_+/\rho}$ and $v:=\sqrt{\rho_-/\rho}$, once we show the elementary inequality
\begin{equation}\label{eq:elementary:log}
	\bra*{u-\frac{1}{v}}^2 + \bra*{ v - \frac{1}{u}}^2 \geq 2 \bra*{\log uv}^2,  \qquad\text{for any } u v > 0.
\end{equation}
By expanding the square and regrouping the elements, we get for the right-hand side the identity
\begin{align*}
    \bra*{u-\frac{1}{v}}^2 + \bra*{ v - \frac{1}{u}}^2 = \bra*{\frac{u}{v}+\frac{v}{u}} \bra*{ uv -2 + \frac{1}{uv}} \geq 8 \abs*{\sinh \bigl( \log \sqrt{uv}\bigr)}^2 \,,
\end{align*}
where we used the elementary bound $a+1/a\geq 2$ for any $a>0$ in the estimate. By recalling, that $\abs*{\sinh x}^2 \geq x^2$ for any $x\in \R$, the proof of~\eqref{eq:elementary:log} is concluded.

To prove the estimate~\eqref{eq:sup:lyp:bound}, we use Jensens inequality and the Poincaré inequality~\eqref{eq:PI:N} from Proposition~\ref{prop:PI-LSI} to estimate
\begin{align*}
	\abs*{ \log \rho^\delta_\kappa - \log \rho^\delta_\lambda}^2 \leq 
	\pra[\bigg]{\delta\sum\nolimits_\kappa \abs*{ \partial^\delta_+ \log \rho^\delta} }^2
	\leq \delta \sum\nolimits_\kappa\abs*{ \partial^\delta_+ \log \rho^\delta}^2 
	\leq C_{\textup{PI}} \delta \sum\nolimits_\kappa\abs*{ \Delta^\delta_\kappa \log \rho^\delta}^2,
\end{align*}
finishing the proof of Lemma~\ref{lem:dDLSSpos}.

\subsection{Proof of Lemma \ref{lem:dDLSShell}}
Let $\rho$ and $\eta$ be two positive solutions to \eqref{eq:dDLSS}. To show that $\dhell(\rho(t),\eta(t))$ is non-increasing in time, we prove non-positivity of its time derivative, using the alternative representation \eqref{eq:dDLSS:savare} and the symmetry of $\ddff$:
\begin{align*}
    \MoveEqLeft -\frac12\frac{\dd}{\dn t}\delta\sum_\kappa\abs[\big]{\sqrt{\rho_\kappa}-\sqrt{\eta_\kappa}}^2
    = -\delta\sum_\kappa\bigl(\sqrt{\rho_\kappa}-\sqrt{\eta_\kappa}\bigr)\biggl(\change{\frac{\dd}{\dn t}\sqrt{\rho_\kappa}-\frac{\dd}{\dn t}\sqrt{\eta_\kappa}}\biggr) \\
    &= \delta\sum_\kappa \bigl(\sqrt{\rho_\kappa}-\sqrt{\eta_\kappa}\bigr)\ddff_\kappa\ddff\big(\sqrt\rho-\sqrt\eta\big) - \delta\sum_\kappa \bigl(\sqrt{\rho_\kappa}-\sqrt{\eta_\kappa}\bigr)\biggl(\frac{\bigl(\ddff_\kappa\sqrt\rho\bigr)^2}{\sqrt{\rho_\kappa}}-\frac{\bigl(\ddff_\kappa\sqrt\eta\bigr)^2}{\sqrt{\eta_\kappa}}\biggr) \\
    &= \delta\sum_\kappa \Bigl(\ddff_\kappa\bigl(\sqrt\rho-\sqrt\eta\bigr)\Bigr)^2 - \delta\sum_\kappa\Bigl(\bigl(\ddff_\kappa\sqrt\rho\bigr)^2+\bigl(\ddff_\kappa\sqrt\eta\bigr)^2\Bigr) \\
    &\qquad + \delta\sum_\kappa \biggl(\sqrt{\frac{\eta_\kappa}{\rho_\kappa}}\bigl(\ddff_\kappa\sqrt\rho\bigr)^2+\sqrt{\frac{\rho_\kappa}{\eta_\kappa}}\bigl(\ddff_\kappa\sqrt\eta\bigr)^2\biggr) \\
	&\change{= \delta\sum_\kappa \biggl(\sqrt{\frac{\eta_\kappa}{\rho_\kappa}}\bigl(\ddff_\kappa\sqrt\rho\bigr)^2- 2\ddff_\kappa\sqrt\rho \, \ddff_\kappa\sqrt\eta  + \sqrt{\frac{\rho_\kappa}{\eta_\kappa}}\bigl(\ddff_\kappa\sqrt\eta\bigr)^2\biggr)} \\
    &= \delta\sum_\kappa \biggl(\sqrt[4]{\frac{\eta_\kappa}{\rho_\kappa}}\ddff_\kappa\sqrt\rho-\sqrt[4]{\frac{\rho_\kappa}{\eta_\kappa}}\ddff_\kappa\sqrt\eta\biggr)^2.
\end{align*}
This concludes the proof of Lemma \ref{lem:dDLSShell}.

\subsection{Proof of Proposition \ref{prp:dQDDx1}}
By Lemma \ref{lem:dDLSSpos}, there is a unique solution $\rho\in C^1(\Rnn,\dprbplus)$ for each initial datum $\rho^0\in\dprbplus$. By classical ODE arguments, the corresponding flow defines a continuous semigroup $\sgrp_+^{\smash[t]{(\cdot)}}$ on $\dprbplus$. By the contractivity property from Lemma \ref{lem:dDLSShell}, at each $t\ge0$, the map~$\sgrp_+^t$ is a globally 1-Lipschitz with respect to $\dhell$. Therefore, there is a unique 1-Lipschitz extension of~$\sgrp_+^t$ to the $\dhell$-closure of $\dprbplus$, which is all of $\dprb$. It is easily checked that the so-defined map~$\sgrp^{\smash{(\cdot)}}:\Rnn\times\dprb\to\dprb$ is continuous and inherits the semigroup property from $\sgrp_+^{\smash[t]{(\cdot)}}$.

\subsection{Proof of Lemma \ref{lem:dDLSSent}}
By taking the time derivative of $\dent$ along a positive solution to \eqref{eq:dDLSS}, we obtain by using the symmetry \eqref{eq:ddffsymmetric} of $\ddff$:
\begin{align*}
    -\frac{\dn}{\dd t}\dent(\rho) 
    &= -\delta\sum\nolimits_\kappa \dot\rho_\kappa\,\log\rho_\kappa
    = 4\delta\sum\nolimits_\kappa \ddff_\kappa\left(\frac{\sqrt{\rho_+\rho_-}-\rho}{\delta^2}\right)\,\log\sqrt{\rho_\kappa} \\
    &= 4\delta^{-1}\sum\nolimits_\kappa \big(\sqrt{\rho_{\kappa+1}\rho_{\kappa-1}}-\rho_\kappa\big)\ddff_\kappa\log\sqrt{\rho_\kappa} \\
    &= 4\delta\sum\nolimits_\kappa \frac{\sqrt{\rho_{\kappa+1}\rho_{\kappa-1}}-\rho_\kappa}{\delta^2} \frac{\log\sqrt{\rho_{\kappa+1}\rho_{\kappa-1}}-  \log \rho_\kappa}{\delta^2}.
\end{align*}
To prove \eqref{eq:dissipation:bound} is thus equivalent to establish the inequality
\begin{align}
    \label{eq:ineq1}
		\sum\nolimits_\kappa \big(\sqrt{\rho_{\kappa+1}\rho_{\kappa-1}}-\rho_\kappa\big) \big(\log\sqrt{\rho_{\kappa+1}\rho_{\kappa-1}}-\log \rho_\kappa\big)
    	\ge \frac14\sum\nolimits_\kappa \big(\sqrt{\rho_{\kappa+1}}+\sqrt{\rho_{\kappa-1}}-2\sqrt{\rho_\kappa}\big)^2
\end{align}
for all $\rho\in\dprbplus$. We change the expression in the sum on the left-hand side by addition of  
\begin{align*}
	0 = \frac12\sum\nolimits_\kappa \left(\frac{(\sqrt{\rho_{\kappa+1}}-\sqrt{\rho_\kappa})^3}{\sqrt{\rho_{\kappa+1}}+\sqrt{\rho_\kappa}}+\frac{(\sqrt{\rho_{\kappa-1}}-\sqrt{\rho_\kappa})^3}{\sqrt{\rho_{\kappa-1}}+\sqrt{\rho_\kappa}}\right),
\end{align*}
which vanishes because of the periodic boundary conditions. After elementary manipulations, \eqref{eq:ineq1} becomes
\begin{align*}
    &\sum\nolimits_\kappa \rho_\kappa\left[\left(\sqrt{\frac{\rho_{\kappa+1}}{\rho_\kappa}}\sqrt{\frac{\rho_{\kappa-1}}{\rho_\kappa}}-1\right) \log\left(\sqrt{\frac{\rho_{\kappa+1}}{\rho_\kappa}}\sqrt{\frac{\rho_{\kappa-1}}{\rho_\kappa}}\right)+\frac{\big(\sqrt{\rho_{\kappa+1}/\rho_\kappa}-1\big)^3}{2\big(\sqrt{\rho_{\kappa+1}/\rho_\kappa}+1\big)}+\frac{\big(\sqrt{\rho_{\kappa-1}/\rho_\kappa}-1\big)^3}{2\big(\sqrt{\rho_{\kappa-1}/\rho_\kappa}+1\big)}\right] \\
	&\ge \frac14\sum\nolimits_\kappa \rho_\kappa\left[\sqrt{\frac{\rho_{\kappa+1}}{\rho_\kappa}}+\sqrt{\frac{\rho_{\kappa-1}}{\rho_\kappa}}-2\right]^2. %
\end{align*}
Validity of this inequality between sums for all $\rho\in\dprbplus$ follows from the respective inequality between corresponding addends in the sums, that is
\begin{align}
  \label{eq:themonster}
  (uv-1)\log(uv) + \frac12\left(\frac{(u-1)^3}{u+1}+\frac{(v-1)^3}{v+1}\right) \ge \frac14(u+v-2)^2
\end{align}
for all real numbers $u,v>0$. Since
\begin{align*}
  \frac{(u-1)^3}{u+1}
  = u^2- 4u+7 - \frac8{u+1},
\end{align*}
and similarly for $v$ in place of $u$,
we have that
\begin{align*}
  \frac{(u-1)^3}{u+1}+\frac{(v-1)^3}{v+1}
  &= u^2+v^2-4u-4v+14  - \frac8{u+1}  - \frac8{v+1} \\
  &= (u+v-2)^2 - 2 uv + 10 - 8\frac{u+v+2}{(u+1)(v+1)} \\
  &= (u+v-2)^2 - 2(uv-1) + 8\frac{uv-1}{uv+u+v+1} ,
\end{align*}
and consequently, \eqref{eq:themonster} is equivalent to
\begin{align}
  \label{eq:themonster2}
  (uv-1)\left[\log(uv) - 1 + \frac4{uv+u+v+1}\right] + \frac14(u+v-2)^2 \ge 0.
\end{align}
Now introduce $\rho>0$ and $\sigma\ge0$ by
\begin{align*}
  \rho:=\sqrt{uv}, \quad \sigma := \sqrt{\frac uv}+\sqrt{\frac vu}-2.
\end{align*}
Substitution of $uv=\rho^2$ and $u+v=(\sigma+2)\rho$ in \eqref{eq:themonster2}
yields the equivalent form
\begin{align}
  \label{eq:themonster3}
  (\rho^2-1)\left[2\log\rho - 1 + \frac4{(\rho+1)^2+\sigma\rho}\right]
  + \left(\left(1+\frac\sigma2\right)\rho-1\right)^2 \ge 0.
\end{align}
We prove \eqref{eq:themonster3} using a case distinction.
\medskip

\noindent
\textbf{Case 1:} $\rho\ge\sqrt e$.

In this case, $\rho^2-1>0$, and since $2\log\rho\ge1$, inequality \eqref{eq:themonster3} follows immediately.
\medskip

\noindent
\textbf{Case 2:} $0<\rho\le1$.

In this case, $\rho^2-1<0$.
It is sufficient to prove that
\begin{align}
  \label{eq:themonsterhelp2}
  1-2\log\rho \ge \frac4{(\rho+1)^2+\sigma\rho}.
\end{align}
Since $\sigma\rho\ge0$, and since
\begin{align*}
  -\log\rho \ge 1-\rho + \frac12(1-\rho)^2
\end{align*}
on the given range of $\rho$'s,
inequality \eqref{eq:themonsterhelp2} follows from
\begin{align*}
  1+2(1-\rho)+(1-\rho)^2 \ge \left(\frac2{1+\rho}\right)^2.
\end{align*}
This last inequality holds since its left-hand side equals $(2-\rho)^2$,
and since
\begin{align*}
  (2-\rho)(1+\rho)=2+\rho-\rho^2= 2 + \rho(1-\rho) >2.
\end{align*}

\noindent
\textbf{Case 3:} $1<\rho<\sqrt e$.

Since $(\rho+1)/2>1$ and $\rho/4<1$ in the specified range,
we can fist split and then estimate the quotient in \eqref{eq:themonster3} as follows:
\begin{align*}
  \frac4{(\rho+1)^2+\sigma\rho}
  = \frac4{(\rho+1)^2} - \frac{4\sigma\rho}{(\rho+1)^2\big[(\rho+1)^2+\sigma\rho\big]}
  \ge \left(\frac2{\rho+1}\right)^2 - \sigma.
\end{align*}
Consequently, \eqref{eq:themonster3} is implied by
\begin{align}
  \label{eq:monsterhelp}
  (\rho^2-1)\left[2\log\rho -1 + \left(\frac2{\rho+1}\right)^2\right]
  + \left[-\sigma(\rho^2-1) + \left(\left(1+\frac\sigma2\right)\rho-1\right)^2\right]  \ge 0.
\end{align}
We prove \eqref{eq:monsterhelp} by showing non-negativity of the two expressions in square brackets.
Concerning the first:
by concavity of the logarithm, and since $\sqrt e<2$,
we have for all $\rho$ in the specified range that
\begin{align*}
  \log\rho \ge \frac12\frac{\rho-1}{\sqrt e-1} \ge \frac12(\rho-1).
\end{align*}
Therefore, the first square bracket in \eqref{eq:monsterhelp} is non-negative if
\begin{align*}
  (\rho-1)(\rho+1)^2 - (\rho+1)^2 + 4 \ge 0,
\end{align*}
which is equivalent to
\begin{align*}
  \rho^3-3\rho+2 \ge 0.
\end{align*}
This is clearly true at $\rho=1$,
and holds for larger values of $\rho$ since the derivative $3(\rho^2-1)$ of the left-hand side is positive for $\rho>1$.
Non-negativity of the second square bracket in \eqref{eq:monsterhelp} is equivalent to
\begin{align*}
  \left(\frac{\sigma^2}4+1\right)\rho^2-(\sigma+2)\rho + (\sigma+1) \ge 0.
\end{align*}
The discriminant of the quadratic polynomial on the left-hand side above is
\begin{align*}
  4(\sigma+1) \left(\frac{\sigma^2}4+1\right)-(\sigma+2)^2 = \sigma^3 > 0.
\end{align*}
which implies non-negativity of the polynomial (even for arbitrary values of $\rho$).

\subsection{Proof of Lemma \ref{lem:dDLSSfish}}
For a given mobility function $\geom(a,b)$, the gradient flow of a functional $\dfnc:\dprbplus\to\R$ with respect to the discretized $L^2$-Wasserstein metric in $\dprb$ is given by
\begin{align}
    \dot\rho_\kappa &=
    -\partial_-^\delta \bra*{\geom(\rho)\,\partial_+^\delta \bra*{\frac{\partial\dfnc}{\partial\rho}}} \notag \\
	&=\delta^{-2}\left(
    \geom(\rho_{\kappa+1},\rho_\kappa)\big[\partial_{\rho_{\kappa+1}}\dfnc(\rho)-\partial_{\rho_{\kappa}}\dfnc(\rho)\big] 
    -\geom(\rho_{\kappa},\rho_{\kappa-1})\big[\partial_{\rho_{\kappa}}\dfnc(\rho)-\partial_{\rho_{\kappa-1}}\dfnc(\rho)\big] 
    \right),\label{eq:W2}
\end{align}
see \cite{Maas2011} or \cite{MaasMatthes} for more details.
For the variational derivative (see also~\eqref{eq:derviv:dentropy} for the scaling in $\delta$) of the discretized Fisher information~\eqref{eq:def:dfish}, we obtain 
\begin{align*}
	\partial_{\rho_\lambda}\dfish(\rho)
	= 2\delta^{-2}\bra*{2 - \frac{\sqrt{\rho_{\lambda+1}}+\sqrt{\rho_{\lambda-1}}}{\sqrt{\rho_\lambda}}} .
\end{align*}
Recalling that $\geom(a,b)=\sqrt{ab}$, we find that
\begin{align*}
    \geom(\rho_{\kappa+1},\rho_\kappa)\big[\partial_{\rho_{\kappa+1}}\dfish(\rho)-\partial_{\rho_{\kappa}}\dfish(\rho)\big]
    &= -2\delta^{-2}\Big[\big(\sqrt{\rho_{\kappa+2}\rho_\kappa}+\rho_\kappa\big) - \big(\sqrt{\rho_{\kappa+1}\rho_{\kappa-1}}+\rho_{\kappa+1}\big)\Big] \\
    &= -2\delta^{-2}\Big[\big(\sqrt{\rho_{\kappa+2}\rho_\kappa}-\rho_{\kappa+1}\big) - \big(\sqrt{\rho_{\kappa+1}\rho_{\kappa-1}}-\rho_{\kappa}\big)\Big].
\end{align*}
The corresponding other term on the right-hand side of \eqref{eq:W2} is obtained by an index shift $\kappa\to\kappa-1$. Using the definition of the discrete Laplacian $\ddff$, we observe that \eqref{eq:W2} turns into
\begin{align*}
    \dot\rho_\kappa = -2\delta^{-2}\ddff_\kappa\big(\sqrt{\rho_+\rho_-}-\rho\big),
\end{align*}
which is \eqref{eq:dDLSS}. Monotonicity of $\dfish$ is now an easy consequence of the representation \eqref{eq:W2} of \eqref{eq:dDLSS}. A summation by parts yields that
\begin{align*}
    -\frac{\dd}{\dn t}\dfish(\rho) 
    &= -\delta\sum\nolimits_\kappa \dot\rho_\kappa\,\partial_{\rho_\kappa}\dfish(\rho) 
    = \delta\sum\nolimits_\kappa \geom(\rho_{\kappa+1},\rho_\kappa)\left[\frac{\partial_{\rho_{\kappa+1}}\dfish(\rho)-\partial_{\rho_{\kappa}}\dfish(\rho)}{\delta}\right]^2 \ge 0.
\end{align*}
Lemma \ref{lem:dDLSSfish} has been proven.

\subsection{Proof of Lemma~\ref{lem:uniform:time:bounds}}
		\change{Thanks to Proposition~\ref{prp:dQDDx1}, it is enough to consider positive solutions by the continuity of the discretized functionals.}
		The estimates are a further application of the crucial a priori bound~\eqref{eq:dissipation:bound}. In combination with Lemma~\ref{lem:convsob1} from the Appendix, we obtain that
		\begin{equation}\label{eq:dfishdent}
			\left(\frac12\dfish(\rhon)\right)^2
			=\left(\delta\sum\nolimits_\kappa\big(\partial^\delta_+\sqrt{\rhon}\big)^2\right)^2 
			\le \delta\sum\nolimits_\kappa \left(\sqrt{\rhon}\right)^2\ \delta\sum\nolimits_\kappa\big(\ddff\sqrt{\rhon}\big)^2 
			\le -\frac{\dd}{\dn t}\dent(\rhon).        
		\end{equation}
		An application of the logarithmic Sobolev inequality~\eqref{eq:LSI:d} with constant $C_{\textup{LSI}}$, leads to the bound
		\begin{align*}
			-\frac{\dd}{\dn t}\dent(\rhon)
			\ge \dent(\rhon)^2 / (4 C_{\textup{LSI}}^2)
		\end{align*}
		The result~\eqref{eq:unident} is now a consequence from the fact that $\dent$ is non-negative, and thus
		\begin{align*}
			\dent(\rhon(t))\le \frac1{\dent(\rhon(0))+t/(4 C_{\textup{LSI}}^2)} \le \frac{4 C_{\textup{LSI}}^2}t.
		\end{align*}
		Next, integration of \eqref{eq:dfishdent} with respect to time in combination with the monotonicity of $\dfish(\rhon)$ yields for every $0<\tau<t$:
		\begin{align*}
			(t-\tau)\dfish(\rhon(t))^2
			\le\int_\tau^t\dfish(\rhon(s))^2\dd s
			\le 4\big(\dent(\rhon(\tau))-\dent(\rhon(t))\big)
			\le 4\dent(\rho(\tau)).
		\end{align*}
		Choosing in particular $\tau=t/2$, and recalling \eqref{eq:unident}, provides
		\begin{align*}
			\dfish(\rhon(t))^2\le 4(t-t/2)^{-1}\dent(\rhon(t/2)) \le 64 C_{\textup{LSI}}^2 t^{-2},
		\end{align*}
		which is \eqref{eq:unidfish}. The last inequality \eqref{eq:unidest} is yet another application of the crucial estimate \eqref{eq:dissipation:bound}. Simply integrate the relation
		\begin{align*}
			\delta\sum\nolimits_\kappa \left(\ddff_\kappa\sqrt\rhon\right)^2      
			\le - \frac{\dd}{\dn s}\dent(\rhon(s)) 
		\end{align*}
		from $s=t>0$ to $s=T>t$ and apply \eqref{eq:unident}.

	For the proof of~\eqref{eq:scheme:shorttime},
	\change{first note that $\dent(\rhon)= \delta \sum_{\kappa} \rhon_\kappa \log \rhon_\kappa$, since $\rhon\in\dprb$.
	Next, fix $p>2$ and observe that the elementary estimate
	\[
		\bra*{\tfrac{p}{2}-1} \log z = \log z^{p/2-1} \leq  z^{p/2-1} -1 \qquad\text{for any $z>0$}
	\]
	implies in particular that 
	\begin{equation}\label{eq:ent:est}
		\bra*{\tfrac{p}{2}-1}\dent(\rhon)\leq \delta \sum\nolimits_\kappa |\rhon_\kappa|^{p/2}-1 \qquad\text{for any $\rhon\in\Prob^\delta$}.
	\end{equation}
    Now let $p$ be specifically of the form $p = 2+\frac{2}{r+1}$ with $r > 0$. The Gagliardo-Nirenberg Sobolev inequalities in Lemma~\ref{lem:dgni}
    provide the estimate $\norm[\big]{\sqrt{\rhon}}_{H^1_N}^2 \geq \norm[\big]{\sqrt{\rhon}}_{L^p_N}^{\frac{2}{\theta}}$  with $\theta=\frac{p-2}{p}$; here we have used that $\norm[\big]{\sqrt{\rhon}}_{L^2_N} = 1$.
	Recalling that $\Delta^\delta = \partial_-^\delta \partial_+^\delta$, and since $\partial_+^\delta \sqrt{\rho^\delta}$ has mean zero, the Poincaré inequality from Proposition~\ref{prop:PI-LSI} now allows to conclude that}
	\begin{align*}
		-4 C_{\textup{PI}} \frac{\dd}{\dn t} \dent(\rhon) 
        &\geq C_{\textup{PI}} \delta \sum\nolimits_\kappa \bra {\ddff \sqrt{\rhon}}^2 
        \geq \delta \sum\nolimits_\kappa \abs{ \partial^\delta_+ \sqrt{\rhon}}^2 
        = \norm[\big]{\sqrt{\rhon}}_{H^1_N}^2 - 1 \\
		&\geq \dnorm{L^p_N}{\sqrt{\rhon}}^{\frac{2}{\theta}}  - 1=\bra[\bigg]{\delta \sum\nolimits_\kappa |\rhon_\kappa|^{p/2}}^{\frac{2}{p\theta}} -1 \change{\stackrel{\eqref{eq:ent:est}}{\geq} \bra*{\bra*{\tfrac{p}{2}-1}\dent(\rhon) + 1}^{\frac{2}{p\theta}}  - 1 \,.}
	\end{align*}
	\change{Note that $\frac{2}{p\theta} = \frac{2}{p-2} = r+1$ and $\tfrac{p}{2}-1=\tfrac{1}{r+1}$. Using that $(1+z)^{r+1} \geq 1+z^{r+1}$ for any $z>0$, we obtain the differential inequality
	\[
		-\frac{\dd}{\dn t} \dent(\rhon)\geq \frac{1}{4 C_{\textup{PI}}} \bra*{ \tfrac{1}{r+1} \dent(\rhon)}^{r+1} \,.
	\]
    Or, equivalently,
    \[
	\frac{\dd}{\dn t} \dent(\rhon)^{-r} \geq C_r := \frac{r}{4 C_{\textup{PI}}(r+1)^{r+1}} \,,
	\]
	which gives via integration in time the bound
	\[
		\dent(\rhon(t)) \leq \bra[\bigg]{C_r t + \frac{1}{(\dent(\rhon(0)))^r}}^{-1/r} \leq \bra*{C_r t}^{-1/r} \,.
	\]
	This is~\eqref{eq:scheme:shorttime}
	}

\section{Continuous limit: Proof of Result~\ref{result:convergence}}
\label{sct:dQDDconv}
For a precise statement of the convergence result, we introduce the reconstruction operator for grid functions $(f_\kappa)_{\kappa\in[N]}$ by piecewise constant embedding,
\begin{equation*}%
	\pid: \mathbb{R}^N \rightarrow L^1(\crc), \quad
	\pid f^\delta(x) = f^\delta_\kappa \quad\text{for $x \in \change{\bigl((\kappa-1/2)\delta, (\kappa+1/2)\delta\bigr)}$}.
\end{equation*}
The adjoint of $\Pi$ is the averaging operator
\begin{equation*}%
	\widehat{\cdot}^\delta:\change{\Prob(\crc)} \rightarrow \mathbb{R}^N, \quad \widehat{\mu}^\delta_{\kappa} = \fint_{\interval_{\kappa}} \change{ \dx\mu}.
\end{equation*}
We discretize the initial data and make it positive using the operator
\begin{equation}\label{eq:discrete:initial}
	\Upsilon^\delta :\change{\Prob(\crc)}\to \dprb \qquad\text{with}\qquad \Upsilon^\delta \rho = (1+\delta)^{-1} \big(\hat\rho^\delta+\delta\big).
\end{equation}
With these notations at hand, we can now formulate the main result of this paper.
\begin{manualtheorem}{A}[Convergence statement]\label{thm:limit}
	Let an initial datum $\rho_0\in \change{\Prob(\crc)}$ be given. For each $\delta>0$, consider the unique global solution $\rhon:[0,\infty)\to\dprb$ to the differential equation \eqref{eq:dDLSS} with initial datum $\rho^\delta(0)=\Upsilon^\delta(\rho_0)$ defined in~\eqref{eq:discrete:initial}.
    
    Then there exists a weakly continuous curve $\rho:[0,\infty)\to \change{\Prob(\crc)}$ with $\rho(0)=\rho_0$ such that along a subsequence $\delta \rightarrow 0$:
	\begin{itemize}
        \item $\pid\rhon(t) \rightarrow \rho(t)$ in \change{$\dot W^{2,\infty}(\crc)$} for every $t\ge0$;
		\item $\pid \change{\sqrt{\rhon}} \rightarrow \change{\sqrt{\rho}}$ strongly in $L^2_\loc((0,\infty); L^2(\crc))$;
		\item $\pid \ddff \sqrt{\rhon} \rightharpoonup \partial_{xx} \sqrt{\rho}$ weakly in $L^2_\loc((0,\infty); L^2(\crc))$;
		\item $\pid \partial_-^\delta \sqrt{\rhon}, \; \pid \partial_+^\delta \sqrt{\rhon} \rightarrow \partial_{x} \sqrt{\rho}$ strongly in $L^2_\loc((0,\infty); L^2(\crc))$.
	\end{itemize}
    Moreover, $\rho$ satisfies the following weak formulation of equation~\eqref{eq:DLSS}:
	\begin{equation}\label{eq:DLSS:weak}
		\int_{0}^{\infty} \int_{\crc} \partial_t \varphi(t,x)\rho(t,x) \dd x \dd t=2  \int_{0}^{\infty} \int_{\crc} \partial_{xx} \varphi(t,x) \bra*{ \sqrt{\rho} \partial_{xx} \sqrt{\rho} - \lvert \partial_x \sqrt{\rho} \rvert^2 } \dd x \dd t
	\end{equation}
	for all $\varphi \in C^\infty_c((0,\infty)\times\crc)$. Finally, if $\fish(\rho_0)<\infty$, then \change{$\Rnn\ni t\mapsto \rho(t,\cdot)$ } is globally H\"older continuous of degree $1/24$ with respect to $\hell$.
\end{manualtheorem}
\begin{proof} 
    The proof is divided into several steps.	
    \medskip

\noindent
\textbf{Step 0. Fundamental a priori bound.}
\newline    
    By Lemma \ref{lem:dDLSSpos}, there is a solution $\rhon$ to our scheme \eqref{eq:dDLSS} with the given positive initial datum $\Upsilon^\delta(\rho_0) \in \dprb$. Recalling the entropy production estimate \eqref{eq:dissipation:bound},
	we infer for all sufficiently small $\delta > 0$ and arbitrary $t>s\ge0$ that
	\begin{align*}
		\int_{s}^{t} \int_{\crc} \Bigr(\pid \ddff \sqrt{\rho^{\smash\delta}(\tau)}\Bigl)^{\!2} \dd x \dd \tau \leq  \dent(\rhon(s)) - \dent(\rhon(t)) \le \dent(\Upsilon^\delta(\rho_0)) %
	\end{align*}
	where we have used monotonicity of $\tau\mapsto\dent(\rhon(\tau))$. If $\ent(\rho_0)$ is finite, then $\dent(\Upsilon^\delta(\rho_0))$ is $\delta$-uniformly bounded, 
	\begin{equation}
        \label{eq:sqrtbyH}
		\dent(\Upsilon^\delta(\rho_0)) \leq \frac{1}{1+\delta} \dent(\widehat{\rho_0}^{\!\delta}) + \frac{\delta}{1+\delta} \dent(1) \leq \dent(\widehat{\rho_0}^{\!\delta}) \leq \ent(\rho_0),
	\end{equation}
    where we have used convexity of $\dent$, and the definition of $\Upsilon^\delta$ in~\eqref{eq:discrete:initial}. If instead $\dent(\rho_0)=+\infty$, then we use \eqref{eq:unident} for an $s$-dependent bound,
    \begin{align}
        \label{eq:sqrtuniversal}
        \int_{s}^{t} \int_{\crc} \Bigr(\pid \ddff \sqrt{\rho^{\smash\delta}(\tau)}\Bigl)^{\!2} \dd x \dd \tau \le 64C_\text{LSI}^2s^{-1}.
    \end{align}
\medskip

\noindent
\textbf{Step 1: Locally uniform Hölder continuity in time.}
\newline
    Let $t>s\ge0$ be given, and fix some $\varphi \in \dot W^{2,\infty}(\crc)$ with $\lVert \partial_{xx} \varphi \rVert_{\infty} \leq 1$. Using summation by parts and Hölder's inequality for sums, we find that
	\begin{align*}
		\MoveEqLeft\int_{\crc} \varphi \left(\pid \rhon(t) - \pid \rhon(s)\right) \dd x
		= \int_{s}^{t} \int_{\crc} \varphi \,\pid \dot\rho^\delta(\tau) \dd \tau \dd x
        = \int_s^t \delta\sum\nolimits_\kappa \hat\varphi^\delta_\kappa \,\dot\rho^\delta_\kappa(\tau)\dd\tau \\
		&= - \int_{s}^{t} \sum\nolimits_\kappa\hat\varphi^\delta_\kappa\, \ddff_\kappa\big(\mob(\tau)\,\ddff \log\rhon(\tau)\big) \dd \tau \\
		&= - \int_{s}^{t} \delta \sum\nolimits_\kappa \ddff_\kappa \hat{\varphi}^\delta \mob_\kappa(\tau) \ddff_\kappa \log\rhon(\tau) \dd \tau\\
		&\leq \biggl( \int_{s}^{t} \tau^{-2/3} \delta \sum\nolimits_\kappa \mob_\kappa(\tau) \left(\ddff_\kappa \hat{\varphi}^\delta\right)^2 \dd \tau \biggr)^{\!1/2} \biggl( \int_{0}^{t} \tau^{2/3}\delta \sum\nolimits_\kappa \mob_\kappa(\tau) \left( \ddff_\kappa \log\rhon(\tau)\right)^2 \dd \tau \biggr)^{\!1/2} \\
        &= \biggl(3(t^{1/3}-s^{1/3})\|\partial_{xx}\varphi\|_{\infty}^2\,\sup_{s<\tau<t}\biggl[\delta\sum\nolimits_\kappa\mob(\tau)\biggr]\biggr)^{1/2}\left(-\int_0^t\tau^{2/3}\frac{\dd}{\dn\tau}\ent(\rhon(\tau))\dd\tau\right)^{1/2}.
    \end{align*}
    Using that $\delta\sum_\kappa\mob \le 3\delta\sum_{\kappa}\rhon\le 3$, see Lemma~\ref{le:threepointmean}, and that 
    \begin{align*}
        \int_0^t\tau^{2/3}f(\tau)\dd\tau
        = \frac23\int_0^t\left[\int_0^\tau \sigma^{-1/3}\dd\sigma\right]f(\tau)\dd\tau
        = \frac23\int_0^t\left[\int_\sigma^t f(\tau)\dd\tau\right]\sigma^{-1/3}\dd\sigma
    \end{align*}
    for any continuous function $f$ on $[0,t]$,
    we conclude that
    \begin{align*}
		\int_{\crc} \varphi \left(\pid \rhon(t) - \pid \rhon(s)\right) \dd x
        &\le \sqrt{6(t^{1/3}-s^{1/3})}\left(\int_0^t\sigma^{-1/3}\ent(\rhon(\sigma))\right)^{1/2}.
    \end{align*}  
	Next, we apply~\eqref{eq:scheme:shorttime} from Lemma \ref{lem:uniform:time:bounds} with $r = 3$ to obtain
	\begin{align*}
		\int_0^t \sigma^{-1/3}\dent(\rhon(\sigma))\dd \sigma  \leq C \max\set*{t^{1/3}, t^{2/3}},
	\end{align*}
    and use the elementary inequality \change{$a^{1/3}-b^{1/3}\le ( a- b)^{1/3}$ for $0\leq b\leq a$}, to finally conclude that
    \begin{align*}
        \int_{\crc} \varphi \left(\pid \rhon(t) - \pid \rhon(s)\right) \dd x
        \le \sqrt{6C}\max\{t^{1/3},t^{1/6}\} (t-s)^{1/6}.
    \end{align*}    
    This shows $\delta$-uniform and local in time Hölder-1/6-continuity of the curves $t\mapsto\pid\rhon(t)$ in the dual of $\dot W^{2,\infty}(\crc)$. We conclude by the generalized Arzelà-Ascoli theorem \cite[Theorem 3.3.3]{AGS} that $\pid\rhon(t)\to\rho(t)$ in $W^{2,\infty}(\crc)'$, locally uniformly with respect to $t\ge0$, with a limit that is Hölder-1/6-continuous.
    \medskip
 
\noindent
\textbf{Step 2. Strong convergence of $\rhon$ in $L^1_\loc((0,\infty);L^1(\crc))$.}
	\newline
    We invoke the refined Aubin-Lions-method from~\cite{RoSa03} to obtain convergence of $\pid \rhon$ to $\rho$ in $L^1_\loc((0,\infty); L^1(\crc))$. 
    Fix $0<\tau<T$. From the key a priori estimate \eqref{eq:sqrtuniversal} in combination with Lemma~\ref{lem:convsob1} and H\"older's inequality, we obtain the following time-integrated bound on the total variation of $\pid\rhon$:
	\begin{align*}
		\int_{\tau}^{T}  \sum\nolimits_\kappa \left|\rhon_{\kappa+1} - \rhon_\kappa\right| \dd t 
        &= \int_{\tau}^{T}  \sum\nolimits_\kappa \Bigl|\sqrt{\rhon_{\kappa+1}} + \sqrt{\rhon_\kappa} \Bigr| \ \Bigl|\sqrt{\rhon_{\kappa+1}} - \sqrt{\rhon_\kappa} \Bigr| \dd t \notag\\
		&\leq \biggl(2\int_\tau^T\delta\sum\big(\rhon_{\kappa+1}+\rhon_\kappa\big)\dd t\biggr)^{1/2} \biggl( \int_{\tau}^{T} \delta \sum \big| \partial^\delta_+ \sqrt{\rhon} \big|^2 \dd t \biggr)^{\!1/2} \\ %
		& \leq 2\sqrt{T} \biggl( \int_{\tau}^{T} \delta \sum \rhon \dd t \biggr)^{\!1/4} \biggl(\int_{\tau}^{T} \delta \sum \big| \ddff \sqrt{\rhon} \big|^2 \dd t\biggr)^{\!1/4} \notag \\
        & \le 2T^{3/4}\sqrt{8C_\text{LSI}}\,\tau^{-1/4}.\notag
	\end{align*}
    From Step 1, we know that $\pid\rhon(t)\to\rho(t)$ in $W^{2,\infty}(\crc)'$ uniformly with respect to $t\in[\tau,T]$. An application of \cite[Theorem 2]{RoSa03} yields convergence $\rhon(t)\to\rho(t)$ in measure with respect to $t\in[\tau,T]$. Since further $\ent(\pid\rhon(t))=\dent(\rhon(t))\le 64C_\text{LSI}^2\tau^{-1}$ and $r\mapsto r(\log r-1)+1$ is of superlinear growth, the convergence is actually in $L^1([\tau,T])$.

    As a consequence, also $\sqrt{\rhon}\to\sqrt\rho$ in $L^2_\loc((0,\infty);L^2(\crc))$.
	\medskip

\noindent
\textbf{Step 3. Weak convergence of the second derivative in $L^2_\loc((0,\infty); L^2(\crc))$.}
	\newline
    By the main a priori estimate \eqref{eq:sqrtuniversal}, there is some $\zeta \in L^2_\loc((0,\infty); L^2(\crc))$ such that for any choice of $0<\tau<T<\infty$, 
	\begin{equation*}
		\pid \ddff \sqrt{\rhon} \rightharpoonup \zeta \quad \text{weakly in $L^2([\tau,T]; L^2(\crc))$}.
	\end{equation*} 
    It remains to show that $\sqrt\rho\in L^2_\loc((0,\infty);H^2(\crc))$ with $\partial_{xx}\sqrt\rho=\zeta$. Let $\varphi\in C^\infty_c((0,\infty)\times\crc)$, and observe that
    \begin{equation}
    \label{eq:dxx}
    \begin{split}
        \int_0^\infty\int_{\crc}\varphi\,\pid\ddff\sqrt{\rhon}\dd x\dd t
        &= \int_0^\infty \delta\sum \hat\varphi^\delta\,\ddff\sqrt{\rhon} \dd t \\
        &= \int_0^\infty \delta\sum \ddff\hat\varphi^\delta\,\sqrt{\rhon}\dd t
        = \int_0^\infty \int_{\crc} \pid \ddff\hat\varphi^\delta\,\pid\sqrt{\rhon}\dd x\dd t.
    \end{split}        
    \end{equation}
    Since $\varphi$ is smooth and of compact support, the second order spatial difference quotients converge uniformly to $\partial_{xx}\varphi$, and consequently,
    \begin{align*}
        \pid\ddff\hat\varphi^\delta \to \partial_{xx}\varphi \quad\text{uniformly as $\delta\to0$}.
    \end{align*}
    Further, $\sqrt{\rhon}\to\sqrt\rho$ in $L^2_\loc((0,\infty);L^2(\crc))$ by Step 2 above. Passing to the limit $\delta\to0$ in \eqref{eq:dxx} yields
    \begin{align*}
        \int_0^\infty\int_{\crc}\varphi\,\zeta\dd x\dd t = \int_0^\infty\int_\crc \partial_{xx}\varphi\,\sqrt\rho\dd x\dd t.
    \end{align*}
    Since $\varphi\in C^\infty_c((0,\infty)\times\crc)$ has been arbitrary, $\zeta$ has been identified as the weak second derivative~$\partial_{xx}\sqrt\rho$.
	\medskip

 \noindent
 \textbf{Step 4. Strong convergence of the first derivatives in $L^2_\loc((0,\infty); L^2(\crc))$.}
    \newline
    We show that $\pid \partial_+^\delta\sqrt{\rhon} \rightarrow \partial_{x} \sqrt\rho$ in $L^2((0,T);L^2(\crc))$; the proof for $\pid \partial_-^\delta\sqrt{\rhon} \rightarrow \partial_{x}\sqrt\rho$ in analogous. The proof goes by interpolation between the convergences obtained in Step 2 and Step 3 above. First, observe that for any $w\in L^2_\loc((0,\infty);H^2(\crc))$, one has
    \begin{align}
        \label{eq:interpol}
        \pid\hat w^\delta \to w, \quad
        \pid\partial_+^\delta\hat w^\delta \to \partial_x w, \quad\text{and}\quad
        \pid\ddff\hat w^\delta \to\partial_{xx}w   
        \quad \text{in $L^2_\loc((0,\infty);L^2(\crc))$}.
    \end{align}
    Applying this specifically to $w:=\sqrt\rho$, it follows from Step 2 and Step 3, respectively, that
    \begin{align*}
        \pid\big[\sqrt{\rhon} - \widehat{\sqrt\rho}{}^\delta\big] \to 0 \quad &\text{strongly in $L^2_\loc((0,\infty);L^2(\crc)\change{)}$}, \\
        \pid\ddff\big[\sqrt{\rhon} - \widehat{\sqrt\rho}{}^\delta\big] \rightharpoonup 0 \quad &\text{weakly in $L^2_\loc((0,\infty);L^2(\crc)\change{)}$}.
    \end{align*}    
    Since weakly convergent sequences are bounded, this implies by means of Lemma \ref{lem:convsob1} that
    \begin{align*}
        \left\|\pid\partial_+^\delta\Big[\sqrt{\rhon}-\widehat{\sqrt\rho}{}^\delta\Big] \right\|_{L^2}^2
        \le \left\|\sqrt{\rhon}-\widehat{\sqrt\rho}{}^\delta\right\|_{L^2}\left\|\pid\ddff\Big[\sqrt{\rhon}-\widehat{\sqrt\rho}{}^\delta\Big] \right\|_{L^2}\to 0,
    \end{align*}
    which means that
    \begin{align*}
        \pid\partial_+^\delta\sqrt{\rhon}-\pid\partial_+^\delta\widehat{\sqrt\rho}^\delta\to0 \quad \text{strongly in $L^2_\loc((0,\infty);L^2(\crc)\change{)}$}.
    \end{align*}
    Recalling \eqref{eq:interpol} again, this yields the desired convergence.
\medskip

\noindent
\textbf{Step 5: Derivation of the weak formulation.} 
\newline
    Recall that the spatially discrete evolution equation \eqref{eq:dDLSS} is
	\begin{equation*}
		\dot{\rho}^\delta = - 2\ddff \biggl( \frac{\sqrt{\rhon_{\smash+}\rhon_{\smash-}} - \rhon}{\delta^2} \biggr).
	\end{equation*} 
    We rewrite the expression inside the second derivative as follows:
	\begin{align*}
		2\biggl( \frac{\sqrt{\rhon_{\smash+} \rhon_{\smash-}} - \rhon}{\delta^2} \biggr) 
        & = \Bigl( \sqrt{\rhon_{\smash+}} + \sqrt{\rhon_{\smash-}} \Bigl) \ddff \sqrt{\rhon} - \biggl( \frac{\sqrt{\rhon_{\smash+}}- \sqrt{\rhon}}{\delta} \biggr)^{\!2} - \biggl( \frac{\sqrt{\rhon}- \sqrt{\rhon_{\smash-}}}{\delta} \biggr)^{\!2} \\
        & = \Big(2\sqrt{\rhon} + \delta\,\partial_+^\delta\sqrt{\rhon} + \delta\,\partial_-^\delta\sqrt{\rhon}\Big)\ddff\sqrt{\rhon} - \big(\partial_+^\delta\sqrt{\rhon}\big)^2 - \big(\partial_-^\delta\sqrt{\rhon}\big)^2
        =: f.
	\end{align*}
    By the convergence properties derived in Step 2, Step 3 and Step 4 above, we have
    \begin{align}
        \label{eq:2weak}
        \pid f \rightharpoonup 2\sqrt\rho\,\partial_{xx}\sqrt\rho - 2\big(\partial_x\sqrt\rho\big)^2 \quad\text{in $L^2_\loc((0,\infty);L^2(\crc))$}.
    \end{align}
    To conclude the derivation of \eqref{eq:DLSS:weak}, let $\varphi\in C^\infty_c((0,\infty)\times\crc)$. An integration by parts in time yields
    \begin{align*}
        \int_0^\infty \int_\crc \partial_t\varphi\,\pid\rhon \dd x\dd t
        &= - \int_0^\infty \int_\crc \varphi\,\pid\dot\rhon\dd x\dd t \\
        &= \int_0^\infty \delta\sum \hat\varphi^\delta\,\ddff f \dd t \\
        &= \int_0^\infty \delta\sum \ddff\hat\varphi^\delta\,f\dd t 
        = \int_0^\infty \int_\crc \pid \ddff\hat\varphi^\delta \pid f \dd x\dd t.
    \end{align*}
    Now pass to the limit $\delta\to0$ with the first and the last integral expressions. Recalling \eqref{eq:interpol} and \eqref{eq:2weak}, the weak formulation \eqref{eq:DLSS:weak} follows.
    \medskip

\noindent
\textbf{Step 6: Hölder continuity in $\dst$ and in $\hell$.}
    If $\ent(\rho_0)<\infty$, or even $\fish(\rho_0)<\infty$, then the Hölder estimate in Step 1 can be improved using, respectively, the better bound \eqref{eq:sqrtbyH} or in addition the monotonicity of $\dfish$ from Lemma  \ref{lem:dDLSSfish} in combination with the interpolation from Lemma~\ref{lem:AubinLions}. 

    Assuming $\ent(\rho_0)<\infty$, we conclude that 
    \begin{align*}
        \int_0^\infty \delta\sum \mob_\kappa(\rhon(\tau)) \left( \ddff_\kappa \log\rhon(\tau)\right)^2 \dd \tau
        = \int_0^\infty \left(-\frac{\dd}{\dn\tau}\dent(\rhon(\tau))\right)\dd\tau
        \le \dent(\rhon(0)) \le \ent(\rho_0).
    \end{align*}
    Fix $t_1>t_0>0$ and introduce $t_s=t_0+s(t_1-t_0)$. We wish to estimate $\ddst_\mob(\rhon(t_0),\rhon(t_1))$. To that end, consider the auxiliary pair $(\eta,\welo)\in\curves$, given by
    \begin{align*}
        \eta^s = \rhon(t_s),
        \quad
        \welo^s = (t_1-t_0)\mob(\eta^s)\ddff\log\eta^s
        \quad \text{for $s\in[0,1]$}.
    \end{align*}
    Indeed, the continuity equation \eqref{eq:dcont999} follows directly from \eqref{eq:dDLSS},
    \begin{align*}
        \frac{\dd}{\dn s}\eta^s 
        = (t_1-t_0)\frac{\dd}{\dn t}\rhon(t_s) 
        = (t_1-t_0)\ddff\big[\mob(\rhon(t_s))\ddff\log\rhon(t_s)\big]
        = \ddff\welo^s .
    \end{align*}
    It thus follows that
    \begin{align*}
        \ddst_\mob(\rho^\delta(t_0),\rho^\delta(t_1))^2
        &\le \int_0^1 \delta\sum\nolimits_\kappa \frac{(\welo^s_\kappa)^2}{\mob_\kappa(\eta^s)} \dd s \\
        &\quad = (t_1-t_0)\int_{t_0}^{t_1} \delta\sum\nolimits_\kappa \mob_\kappa(\rhon(t)) \left( \ddff_\kappa \log\rhon(t)\right)^2 \dd t
        \le \dent(\rho_0)\ (t_1-t_0).
    \end{align*}
    In conclusion, we obtain the global Hölder estimate
    \begin{align}
        \label{eq:dstholder}
        \ddst_\mob(\rhon(t_0),\rhon(t_1)) \le \sqrt{\ent(\rho_0)} (t_1-t_0)^{1/2}.   
    \end{align}
    Now assume in addition that $\fish(\rho_0)<\infty$. We then obtain a $\delta$-uniform estimate on the initial discrete Fisher information; first, notice that 
    \begin{align*}
	   \dfish(\Upsilon^\delta(\rho_0)) = \frac{1}{1+\delta} \dfish(\widehat {\rho_0}^\delta)
    \end{align*}
    by a scaling argument. Second, we shall prove below that 
    \begin{align}
        \label{eq:dfishbound}
        \dfish(\widehat{\rho_0}^\delta) \le 8\fish(\rho_0).
    \end{align}
    Then $\delta$-uniform bound on $\dfish(\rhon(0))$ is propagated to all later times by monotonicity of $t\mapsto\dfish(\rhon(t))$, see Lemma \ref{lem:dDLSSfish}. The uniform Hölder bound in $\dhell$ is then a consequence of the uniform Hölder bound~\eqref{eq:dstholder} in~$\dst$ above and the interpolation Lemma \ref{lem:AubinLions:discrete}. Since
    \[ \dhell(\rhon(t_0),\rhon(t_1)) = \hell\big(\pid\rhon(t_0),\pid\rhon(t_1)\big), \]
    and $\dhell$ is weakly lower semi-continuous, the Hölder estimate passes to the limit $\delta\to0$.
    
    It remains to prove \eqref{eq:dfishbound}. Note that the function $(a,b)\mapsto (\sqrt{a}-\sqrt{b})^2$ is jointly convex, and hence Jensen's inequality can be applied to conclude
    \begin{align*}
	   \abs[\Big]{\sqrt{\widehat\rho^\delta_{\kappa+1}}-\sqrt{\widehat\rho^\delta_{\kappa}}}^2 &= \abs[\Bigg]{\sqrt{ \fint_{\interval_{\kappa}} \fint_{\interval_{\kappa+1}} \rho(x) \dd x \dd y} - \sqrt{  \fint_{\interval_{\kappa}} \fint_{\interval_{\kappa+1}} \rho(y) \dd x \dd y}}^{2} \\
	   &\leq  \fint_{\interval_{\kappa}} \fint_{\interval_{\kappa+1}}\abs*{ \sqrt{\rho(x)} - \sqrt{\rho(y)}}^2 \dd x \dd y \\
	   &\leq \fint_{\interval_{\kappa}} \fint_{\interval_{\kappa+1}}\abs*{ \int_x^y \partial_z \sqrt{\rho(z)}\dd z}^2 \dd x \dd y 
	   \leq 2 \delta \int_{\interval_{\kappa}\cup \interval_{\kappa+1}} \abs*{ \partial_z \sqrt{\rho(z)}}^2 \dd z. 
    \end{align*}
	We conclude the proof of the bound~\eqref{eq:dfishbound} by
    substituting the bound in the Definition of $\dfish$ from~\eqref{eq:def:dfish}
    \begin{equation*}
	   \dfish(\widehat \rho^\delta) \leq 4 \sum\nolimits_\kappa\int_{\interval_{\kappa}\cup \interval_{\kappa+1}} \abs*{ \partial_z \sqrt{\rho(z)}}^2 \dd z \leq 8 \fish(\rho) . \qedhere
    \end{equation*}
\end{proof}

\appendix

\section{Functional inequalities on the discretized circle}

We need a basic interpolation estimate, resembling an interpolation of $\dot H^1$ with $L^2$ and $\dot H^2$ on the discrete level.
\begin{lemma}\label{lem:convsob1}
	Let $N \in \mathbb{N}$ and set $\delta = N^{-1}$. Then any $v,w \in \mathbb{R}^N$ satisfy the discrete interpolation estimate
	\begin{align*}
		\delta \sum\nolimits_\kappa \partial_+^\delta v \, \partial_+^\delta w \leq \biggl( \delta \sum\nolimits_\kappa \left( \ddff v \right)^2 \biggr)^{\!1/2} \biggl( \delta \sum\nolimits_\kappa w^2 \biggr)^{\!1/2} . 
	\end{align*}
\end{lemma}
\begin{proof}%
	We expand the square and reorder the terms to finally estimate with the Cauchy-Schwarz inequality
	\begin{align*}
		\delta \sum\nolimits_\kappa \frac{v_+ - v}{\delta}\frac{w_+ - w}{\delta} & = \delta \sum\nolimits_\kappa \pra[\bigg]{\frac{v_+ - v}{\delta^2}w_+ - \frac{v_+ - v}{\delta^2}w }
		 =  \delta \sum\nolimits_\kappa \pra[\bigg]{\frac{v - v_-}{\delta^2}w - \frac{v_+ - v}{\delta^2}w }\\
		&=  \delta \sum\nolimits_\kappa \frac{v_+ + v_- - 2v}{\delta^2} (- w)
		 \leq \biggl( \delta \sum\nolimits_\kappa \biggl(\frac{v_+ + v_- - 2v}{\delta^2}\biggr)^{\!2} \biggr)^{\!1/2} \biggl( \delta \sum\nolimits_\kappa w^2 \biggr)^{\!1/2} . \qedhere
	\end{align*}
\end{proof}
\begin{lemma}[Discrete Gagliardo-Nirenberg-type inequality]\label{lem:dgni}
	Let $N\in \mathbb{N}$, $p\geq 2$ and set $\theta = \frac{p-2}{p}$. Then any $v \in  \mathbb{R}^N$ satisfies, with the notations from \eqref{def:LN},
	\begin{equation*}
	\dnorm{L^p_N}{v} \leq \dnorm{L^2_N}{v}^{1-\theta} \dnorm{H^1_N}{v}^\theta.
	\end{equation*}
\end{lemma}
\begin{proof}
	Let $v \in \mathbb{R}^N$ and $m \in [N]$ such that $v_m = \dnorm{\infty}{v}$. Denoting $\bar{v} := \delta \sum_{\kappa\in [N]} v_\kappa$ and choosing $j \in [N]$ such that $|v_j - v_m| \geq |\bar{v} - v_m|$ we infer
	\begin{align*}
	\dnorm{\infty}{v} \leq |\bar{v}| + |v_m - \bar{v}| \leq |\bar{v}| + |v_m - v_j| %
    \leq \delta \sum\nolimits_\kappa |v_\kappa|+ \delta \sum\nolimits_\kappa |\partial_+^\delta v| \leq \dnorm{H^1_N}{v} . 
	\end{align*}
	With this we deduce
	\begin{align*}
		\dnorm{L^p_N}{v} &= \bra[\bigg]{\delta\sum\nolimits_\kappa \abs v_\kappa^{2} \abs v_\kappa^{p-2}}^{\!1/p} \leq \dnorm{L^2_N}{v}^{\frac 2p} \dnorm{L^\infty_N}{v}^{\frac{p-2}{p}} \leq \dnorm{L^2_N}{v}^{\frac 2p} \dnorm{H^1_N}{v}^{\frac{p-2}{p}} . \qedhere
	\end{align*}
\end{proof}

\begin{proposition}[Poincaré and logarithmic Sobolev inequality on the discrete torus]\label{prop:PI-LSI}
      For any $N\geq 2$ ($\delta=N^{-1}$) the discrete \emph{Poincaré inequality} holds
      \begin{equation}\label{eq:PI:N}
            \delta\sum\nolimits_\kappa f_\kappa^2 \leq C_{\textup{PI}} \delta \sum\nolimits_\kappa \abs*{ \partial^\delta_+ f}^2  \quad\text{with}\qquad C_{\textup{PI}} =  \frac{\delta^2}{2(1-\cos(2\pi \delta))} = \frac{1}{4\pi^2} + O(\delta^2) \leq \frac{1}{16},
      \end{equation}
      for any $f:[N]\to \R$ with $\delta \sum\nolimits_\kappa f_\kappa=0$.

      Likewise, the discrete \emph{logarithmic Sobolev inequality} holds 
      \begin{equation}\label{eq:LSI:N}
          \delta\sum\nolimits_\kappa f_\kappa^2 \log f_\kappa^2 \leq \frac{25}{16\pi^2} \delta \sum\nolimits_\kappa \abs*{ \partial^\delta_+ f}^2 ,
      \end{equation}
      for any $f:[N]\to \R$ with $\delta \sum\nolimits_\kappa f_\kappa^2 = 1$, which equally stated in terms of the discrete entropy $\dent$ from~\eqref{eqdef:dentropy} and discrete Fisher information $\dfish$ from~\eqref{eq:def:dfish} becomes
      \begin{equation}\label{eq:LSI:d}
            \dent(\rho^\delta) \leq C_{\textup{LSI}} \dfish(\rho^\delta) \qquad\text{with}\qquad C_{\textup{LSI}}=\frac{25}{8\pi^2}
      \end{equation}
      for any $\rho^\delta\in \dprb$.
\end{proposition}
\begin{proof}
  We refer to~\cite[§4.2]{Diaconis--Saloff-Coste1996} using the connection to Markov chains, which we briefly explain. 
  \change{This is only necessary for the logarithmic Sobolev inequality, since the optimal Poincaré constant in our case follows from its connection to eigenvalues of circulant matrices, which are explicit~\cite{Gray2006}.}
  
  The left-hand side in~\eqref{eq:PI:N} is by the mean value assumption on $f$ equal to the variance with respect to the uniform measure $\pi_\kappa = \delta$ for $\kappa\in [N]$ defined for any $g:[N]\to \R$ by
  \begin{equation*}
      \operatorname{Var}(g) = \frac{1}{2} \sum\nolimits_{\kappa,\lambda} \bra*{g_\kappa- g_\lambda}^2 \pi_\kappa \pi_\lambda = \delta \sum\nolimits_\kappa \pra[\Big]{ g_\kappa - \delta\sum\nolimits_{\lambda} g_\lambda}^2 =\delta \sum\nolimits_\kappa f_\kappa^2 ,
  \end{equation*}
  where $f_\kappa = g_\kappa -\delta\sum_\lambda g_\lambda$ has average zero.

  Similarly, the right-hand side of~\eqref{eq:PI:N} and~\eqref{eq:LSI:N} are related to the Dirichlet form of the symmetric random walk on $[N]$ with jump kernel $K(\kappa,\kappa\pm 1)=\frac{1}{2}$ and zero else. Indeed, we find
  \begin{equation*}
      \mathcal{D}(f,f) = \frac{1}{2} \sum\nolimits_{\kappa,\lambda} \bra*{f_\kappa-f_\lambda}^2 K(\kappa,\lambda)\pi_\kappa =  \frac{1}{2} \delta\sum\nolimits_\kappa \abs*{ f_\kappa - f_{\kappa+1}}^2 = \frac{\delta^2}{2} \delta \sum\nolimits_\kappa \abs*{ \partial^{\delta}_+ f}^2 
  \end{equation*}
  introducing a factor of $\frac{\delta^2}{2}$ in comparison to the right-hand sides of~\eqref{eq:PI:N} and~\eqref{eq:LSI:N}. 

  The spectral gap $\lambda$ defined by $\lambda=\min\set*{ \mathcal{D}(f,f)/\operatorname{Var}(f): \operatorname{Var}(f)\neq 0}$ is given explicitly by $\lambda = 1-\cos\frac{2\pi}{N} \sim \frac{2 \pi^2}{N^2} + O(N^{-4})$ \change{(see e.g.~\cite[(3.7)]{Gray2006}).}
  Hence, by the scaling of the Dirichlet form noted above, we get the estimate~\eqref{eq:PI:N}. 

  Likewise, the left-hand side of~\eqref{eq:LSI:N} is the relative entropy of the measure $f^2 \pi$ with respect the uniform measure on $[N]$ defined by
  \[
    \operatorname{Ent}_\pi(f^2) = \delta \sum\nolimits_\kappa f_\kappa^2 \log \frac{f_\kappa^2}{\delta\sum_\lambda f_\lambda^2}  = \delta \sum\nolimits_\kappa f_\kappa^2 \log f_\kappa^2 ,
  \]
  provided $f^2$ is normalized such that $\delta\sum_\lambda f_\lambda^2=1$.
  Then, the logarithmic Sobolev constant is defined by $\alpha=\min\set*{ \mathcal{D}(f,f)/\operatorname{Ent}(f^2): \operatorname{Ent}(f^2)\neq 0}$. The result in~\cite[§4.2]{Diaconis--Saloff-Coste1996} implies that $\alpha \geq 8\pi^2/(25 N^2)$, which by the scaling of the Dirichlet form translates to the claimed result~\eqref{eq:LSI:N}. 
\end{proof}

\bibliographystyle{abbrv}
\bibliography{main}	

\end{document}